\documentclass{article}

\usepackage{amsmath}
\usepackage{amsfonts}
\usepackage{amsthm}
\usepackage{ifthen}
\usepackage{accents}
\usepackage{verbatim}
\usepackage{graphicx}
\usepackage{subfigure}
\usepackage{hyperref}

\newcommand{\bx}{\mathbf x}
\newcommand{\bn}{\mathbf n}
\newcommand{\be}{\mathbf e}
\newcommand{\by}{\mathbf y}
\newcommand{\bz}{\mathbf z}
\newcommand{\R}{\mathbb R}
\newcommand{\bff}{\mathbf f}

\newcommand{\ve}{\varepsilon}

\DeclareMathOperator{\tr}{tr}

\let\phi=\varphi
\def\eps{\varepsilon}

\def\id{I} %{\mathrm{id}}

\def\reals{\mathord{\mathbb{R}}}

\makeatletter
\def\abs#1{{\mathopen{\vert} #1 \mathclose{\vert}}}
\newcommand\Abs[2][DEFAULT]%
{%
  \ifthenelse{\equal{#1}{DEFAULT}}%
    {\left\vert #2 \right\vert}%
    {\@nameuse{#1l}\vert #2 \@nameuse{#1r}\vert}%
}
\def\norm#1{{\mathopen{\Vert} #1 \mathclose{\Vert}}}
\newcommand\Norm[2][DEFAULT]%
{%
  \ifthenelse{\equal{#1}{DEFAULT}}%
    {\left\Vert #2 \right\Vert}%
    {\@nameuse{#1l}\Vert #2 \@nameuse{#1r}\Vert}%
}
\makeatother

\newtheorem{theorem}{Theorem}[section]
\newtheorem{corollary}{Corollary}[section]
\newtheorem{statement}{Statement}[section]
\newtheorem{lemma}{Lemma}[section]
\newtheorem{assumption}{Assumption}[section]
\newtheorem{proposition}{Proposition}[section]

\theoremstyle{definition}
\newtheorem{definition}{Definition}[section]

\theoremstyle{remark}

\DeclareMathOperator{\sgn}{sgn}
\DeclareMathOperator*{\argmin}{arg\,min}

\title{Topological methods in analysis of periodic and chaotic
canard-type trajectories %in a class of multi-dimensional systems
}

\author{
A.~V.~Pokrovskii \vspace{1mm} \\
\sl \small Department of Applied Mathematics, \\
\sl \small University College, Cork, Ireland \vspace{1mm} \\
\sl \small E-mail: a.pokrovskii@ucc.ie \vspace{2mm} \\
\and  A.~A.~Pokrovskiy \vspace{1mm}\\
\sl \small Department of Applied Mathematics, \\
\sl \small Girton College, Cambridge, UK \vspace{1mm} \\
\sl \small E-mail: ap468@hermes.cam.ac.uk \vspace{2mm} \\
\and A.~Zhezherun \vspace{1mm} \\
\sl \small Department of Civil and Environmental Engineering, \\
\sl \small University College, Cork, Ireland \vspace{1mm} \\
\sl \small E-mail: a.zhezherun@mars.ucc.ie \vspace{2mm} \\
}
\date{}

\begin{document}

\begin{center}
\begin{minipage}{12cm}
\maketitle
\end{minipage}
\end{center}

%\newpage

\section{Introduction\label{intro}}
This paper investigates the role of topological methods in the
analysis of canard-type periodic and chaotic trajectories. In
Sections \ref{intro} -- \ref{p:g} we apply topological degree
\cite{dei, geom} to the analysis of multi-dimensional canards. This
part of the paper was written mainly by the first and the last
authors. Sections \ref{twod} -- \ref{last} are devoted to an
application of a special corollary of the Poincar\'e--Bendixson
theorem to the existence of periodic two-dimensional canards. This
fragment was written mainly by the first and the second authors.

%Topological degree \cite{dei, geom} is one of
%the principal toolboxes of the modern theory of nonlinear dynamical
%systems. The range of applications of this toolbox is rapidly
%growing, see, for instance, \cite{mawhin}. We discuss in this paper
%a new, to the best of our knowledge, scheme of applying topological
%degree in the analysis of canard-type trajectories.

If $W \colon \reals^d \to \reals^d$ is a continuous mapping, $\Omega
\subset \reals^d$ is a bounded open set, and $y \in \reals^d$ does
not belong to the image $W(\partial \Omega)$ of the boundary
$\partial \Omega$ of $\Omega$, then the symbol $\deg(W, \Omega, y)$
denotes the \emph{topological degree} \cite{dei} of $W$ at $y$ with
respect to $\Omega$. If $0 \not\in W(\partial \Omega)$, then the
integer number $\gamma(W, \Omega) = \deg(W, \Omega, 0)$, called the
\emph{rotation of the vector field $W$ at $\partial \Omega$}, is
well defined. A detailed description of properties of the number
$\gamma(W, \Omega)$ can be found, for example, in \cite{geom}. In
particular, if $\id$ denotes the identity mapping, $\id(x) \equiv
x$, then the number $\gamma(\id - W, \Omega)$ measures the algebraic
number of fixed points of the mapping $W$ in $\Omega$.

%We apply topological degree to investigate canard-type periodic
%trajectories of singularly perturbed differential equations.
%Let us recall some related terminology.
Consider the slow-fast system
\begin{equation} \label{e:main}
{\arraycolsep=0pt \begin{array}{rl}
\dot x &{} = X(x, y,\eps) + %\delta
\hat X(x, y, z, \eps), \\[3pt]
\eps\dot y &{} = Y(x, y, \eps) + %\delta
\hat Y(x, y, z, \eps), \\[3pt]
\dot z &{} = Z(x, y, z, \eps).
\end{array} }
\end{equation}
Here
\[
x \in \reals^2, \quad y \in \reals^1, \quad z \in \reals^d,
\]
and $\eps>0$ is a small parameter. The terms $\hat X(x, y, z, \eps)$
and $\hat Y(x, y, z, \eps)$ are supposed to be small with respect to
the uniform norm:
\[
\sup\abs{\hat X(x, y, z, \eps)}, \sup\abs{\hat Y(x, y, z, \eps)}
\ll 1.
\]
{\em However, no estimates on derivatives of those functions are
assumed.}
%For $\delta=0$ we have the three dimensional slow-fast system
%\begin{equation} \label{e:main_red} {\arraycolsep=0pt
%\begin{array}{rl}
%\dot x &{} = X(x, y,\eps) %+ \delta \hat X(x, y, z, \eps),
%\\[3pt]
%\eps\dot y &{} = Y(x, y, \eps) %+ \delta \hat Y(x, y, z, \eps)
%\dot z &{} = Z(x, y, z, \eps).
%\end{array} }
%\end{equation}
%and an additional slow equation for $z$. In this case $z$ is
%affected by dynamics of $x, y$, but $x$ and $y$ vary independently
%from $z$. For $\delta>0$ the variable $z$ influences dynamics of the
%pair $x,y$; for $\delta<< 1$ this coupling is weak. This is a
%typical situation when the influence of $z$ is indirect.

%Informally speaking, we
%have `moderate' couplings between the components of $x$, `fast'
%coupling $x \to y$, and relatively slow couplings $z \to x, y$.
%For instance, $x$ are predators, $y$ is prey, and $z$ are other
%species that are affected indirectly.

%The functions in the right hand side are supposed to be smooth.
The subset
\begin{equation} \label{mmE}
S = \left\{ (x, y, z) \in \reals^2 \times \reals^1 \times \reals^d
\colon Y(x, y, 0) = 0 \right\}
\end{equation}
of the phase space is called a \emph{slow surface} of the system
\eqref{e:main}: on this surface the derivative $\dot y$ of the
fast variable is zero, the small parameter $\eps$ vanishes,
and there are no disturbances $\hat X$ and $\hat Y$. %, (\ref{E:no122}).
The part of $S$ where
\begin{equation} \label{arE}
Y'_y(x, y, 0) < 0 \quad \bigl( > 0 \bigr)
\end{equation}
is called \emph{attractive} (\emph{repulsive}, respectively). The
subsurface $L \subset S$ which separates attractive and repulsive
parts of $S$ is called a \emph{turning subsurface}.
%In what follows, we suppose the turning subsurface to be smooth.

Trajectories which at first pass along, and close to, an attractive
part of $S$ and then continue for a while along the repulsive  part
of $S$ are called \emph{canards} or \emph{duck-trajectories}.

We apply topological decree to prove existence, and to locate
with a given accuracy periodic and chaotic canards of system
\eqref{e:main}. The canards which may be found in this way are
topologically robust: they vary only slightly if the right hand
side of the system is disturbed. This does not mean that the canard
trajectories are stable in Lyapunov sense. However, unstable
periodic canards are useful on their own. Whenever an (unstable)
periodic canard describes processes which are interesting, for instance,
from the technological point of view, this process can stabilized using
standard feedback control algorithms (note the role of the Pyragas control
in this area). Topologically robust chaotic canards also
have a role: their existence implies the existence of an infinite ensemble
of (unstable) periodic canards. General features of this ensemble
and methods of accurate localization of each of its members follow
from our constructions below. Thus, in this case one has a wide
choice of possible periodic modes in system \eqref{e:main},
each of which may be stabilized in the usual way.

\section{Periodic canards} \label{s:main1}

In this section we formulate the main existence result for
topologically stable canard-type periodic trajectories.

\begin{assumption} \label{a:smooth}
We suppose that the function $X$ and $Z$ are continuously
differentiable, and the function $Y$ is twice continuously
differentiable. The functions $\hat X$ and $\hat Y$ are continuous.
\end{assumption}

Emphasize again, that we do not require any smoothness of the
functions $\hat X$ and $\hat Y$. In particular, if there is no
variable $z,$ then we investigate existence of (periodic) canards of
the three dimensional system
\begin{equation} \label{e:main_red}
{\arraycolsep=0pt \begin{array}{rl} \dot x &{} = X(x, y,\eps) +
%\delta
\hat X(x, y, \eps),
\\[3pt]
\eps\dot y &{} = Y(x, y, \eps) +
%\delta
\hat Y(x, y, \eps)
%\dot z &{} = Z(x, y, z, \eps).
\end{array} }
\end{equation}
Here the disturbances $\hat X(x, y, \eps)$ and $\hat Y(x,y,\eps)$
are small in the uniform norm, but there are no bounds on their
derivatives. Even in this three-dimensional situation applicability
of the standard toolboxes, which are based on asymptotical
representations of slow integral manifolds, is questionable.
%Some
%generalizations will be briefly discussed later.

Loosely speaking, we prove that {\em a periodic canard in disturbed
system \eqref{e:main} exists, providing existence of a periodic
canard in the undisturbed system}
\begin{equation} \label{e:main_0}
{\arraycolsep=0pt \begin{array}{rl}
\dot x &{} = X(x, y, \eps), \\[3pt]
 \eps\dot y &{} = Y(x, y, \eps).
\end{array} }
\end{equation}

A point $(x_c, y_c)$ is called a \emph{critical point} of the
system \eqref{e:main_0}, if it satisfies the equations
\begin{alignat}{1}
\langle X(x_c, y_c, 0), Y'_x(x_c, y_c, 0) \rangle &= 0, \label{3E} \\
Y(x_c, y_c, 0) &= 0 \label{4E}, \\
Y_y(x_c, y_c, 0) &= 0.  \label{5E}
\end{alignat}
This is a system of three equations with three variables, so in
general case it is expected to have solutions. The existence of
critical points is important for the phenomenon of canard type
solutions, because every canard, which first goes along the stable
slow integral manifold for $t < 0$ and then along the unstable slow
integral manifold for $t > 0$, must pass through a small vicinity of
a critical point. Thus, we turn our attention to the critical
points, and to the behavior of \eqref{e:main_0} in a vicinity of
such a point.

A critical point is called \emph{non-degenerate}, if the following
inequalities hold:
\begin{alignat}{1}
X(x_c, y_c, 0) &\not= 0, \label{nev1} \\
Y'_x(x_c, y_c, 0) &\not= 0, \label{nev2} \\
Y''_{yy}(x_c, y_c, 0) &\not= 0. \label{nev3}
\end{alignat}
We consider only non-degenerate critical points. Note that
non-degeneracy is stable with respect to small perturbations of
the right-hand side of \eqref{e:main_0}.

To study periodic and chaotic canards of the full system
\eqref{e:main}, we first consider canards passing through a small
vicinity of a critical point $(x_c, y_c)$ of the truncated system
\eqref{e:main_0}. Without loss of generality we assume that
the critical point is situated at the origin:
\begin{equation} \label{orE}
x_c = y_c = 0.
\end{equation}

Consider the auxiliary system
\begin{equation} \label{e:attr}
{\arraycolsep=0pt \begin{array}{rl}
\dot x &{} = X (x, y, 0), \\[3pt]
\langle (\dot x, \dot y), Y'_x(x, y) \rangle &{} = 0.
\end{array} }
\end{equation}
If the initial point $(x_0, y_0)$ lies on the slow surface
\begin{equation} \label{mm0E}
S_0 = \left\{ (x, y) \in \reals^{3} \colon Y(x, y, 0) = 0 \right\}
\end{equation}
of the system \eqref{e:main_0}, then \eqref{e:attr} is equivalent to
\begin{equation} \label{e:attr3}
{\arraycolsep=0pt \begin{array}{rl}
\dot x &{} = X (x, y, 0), \\[3pt]
(x, y) &{} \in S_0.
\end{array} }
\end{equation}
The system \eqref{e:attr3} is important because it describes
the singular limits of the solutions of \eqref{e:main_0} which
lie on the slow surface. Equations \eqref{e:attr} can also be
rewritten in the following form:
\begin{equation} \label{e:attr2}
{\arraycolsep=0pt \begin{array}{rl}
\dot x &{} = X(x, y, 0), \\[3pt]
\dot y Y_y(x, y, 0) &{} = -\langle X (x, y, 0), Y_x(x, y, 0) \rangle.
\end{array} }
\end{equation}
Due to \eqref{5E} this system has a singularity at the origin. Therefore,
the existence and uniqueness of a solution of \eqref{e:attr2} which
starts at the origin requires an additional assumption and will
be discussed in detail later.

To describe the dynamics of the system \eqref{e:main_0} near
the origin, we introduce a special coordinate system $(x^{(1)},
x^{(2)}, y)$ in the three-dimensional space of pairs $(x, y)$.
We choose $x^{(1)}$ to be co-directed with the gradient
$Y'((0, 0), 0, 0)$, and $x^{(2)}$ to be orthogonal to $x^{(1)}$ and $y$.
In the coordinate system $(x^{(1)}, x^{(2)}, y)$ the gradient of
$Y(x^{(1)}, x^{(2)}, y, 0)$ at the origin takes the form
\begin{equation} \label{grE}
Y'(0, 0, 0, 0) = (\xi, 0, 0), \quad \xi > 0,
\end{equation}
and the equation \eqref{e:main_0} takes the form
\begin{equation} \label{e:main_1}
{\arraycolsep = 0pt \begin{array}{rl}
\dot x^{(1)} &{} = X^{(1)}(x^{(1)}, x^{(2)}, y, 0), \\[3pt]
\dot x^{(2)} &{} = X^{(2)}(x^{(1)}, x^{(2)}, y, 0), \\[3pt]
\dot y &{} = Y(x^{(1)}, x^{(2)}, y, 0).
\end{array} }
\end{equation}
Equations \eqref{grE} and \eqref{3E} imply
\begin{equation} \label{e:x1at0}
X^{(1)}(0, 0, 0, 0) = 0.
\end{equation}
Taking into account the non-degeneracy of the origin, we can guarantee
the inequalities
\begin{equation} \label{geE}
X^{(2)}(0, 0, 0, 0) > 0, \qquad
Y''_{yy}(0, 0, 0, 0) = \zeta > 0,
\end{equation}
by changing, if necessary, the directions of the $x^{(2)}$ and $y$ axes.

The existence of canards and uniqueness of solutions of \eqref{e:attr2}
is guaranteed by the following assumption.
\begin{assumption} \label{ineA}
\begin{gather}
2 X^{(1)}_{x^{(2)}} (0, 0, 0, 0) Y''_{yy}(0, 0, 0, 0) -
X^{(1)}_{y}(0, 0, 0, 0) Y''_{x^{(2)}y}(0, 0, 0, 0) < 0, \label{olE} \\
X^{(1)}_{y}(0, 0, 0, 0) > 0. \label{ol1E}
\end{gather}
\end{assumption}

\begin{lemma} \label{existL}
There exist $T_{a} < 0 < T_{r}$ such that system \eqref{e:attr} with
the initial conditio
\[
x(0) = y(0) = 0,
\]
has the unique solution
\[
w^*(t) = (x^*(t), y^*(t)), \quad T_a < t < T_r,
\]
and the inequalities
\begin{alignat}{2}
Y_{y} (x^*(t), y^*(t)) &{}> 0,
&\quad 0 < {}& t < T_{r}, \label{e:l1gtz} \\
Y_{y} (x^*(t), y^*(t)) &{}< 0,
&\quad T_{a} < {}& t < 0 \label{e:l1ltz}
\end{alignat}
hold. In other words, the half-trajectory
$\left(x^{*}(t),y^{*}(t)\right)$, $T_a < t < 0$
lives on the attractive part of the slow surface \eqref{mm0E},
and the half-trajectory
$\left(x^{*}(t),y^{*}(t)\right)$, $0 < t < T_r$
lives on the repulsive part.
\end{lemma}

Lemma \ref{existL} implies strict limitation on the possible
location of canards of system \eqref{e:main_0} that passing near the
origin: such canards should follow closely the solution $w^{*}(t)$
for a certain interval $t_a < 0 < t_r$.
%For small $\delta$ this
%implies also that the 3-D projection of a canard of the full system
%\eqref{e:main} must behaive in the same way.
%Using standard tools, see [?], it is easy to see that for any
%time interval $[t_a, t_r] \subset (T_a, T_r)$ with $t_a < 0 < t_r$
%such a canard exists for every sufficiently small $\eps, \delta > 0$.
%
%The question whether there exist \emph{periodic} canards is less obvious.
The above argument shows that a periodic canard should have a segment of fast
motion from a small neighborhood of some point of the repulsive part of
$w^*(t)$ to a small neighborhood of the attractive part of $w^*(t)$;
this fast motion is, consequently, almost vertical (i.e., almost
parallel to the $y$ axis). More precisely, if there is a limit of
periodic canards as $\eps \to 0$, then the limiting closed
curve has necessarily a vertical segment connecting the repulsive
and attractive parts of $w^*(t)$. The next assumption ensures
a possibility of such vertical jumps.

\begin{assumption} \label{a:intersect}
The two-dimensional curves $\Gamma_a$, $\Gamma_r$ defined by
\[
\Gamma_a = \left\{ x^*(t) \colon T_a < t < 0 \right\}, \quad
\Gamma_r = \left\{ x^*(t) \colon 0 < t < T_r \right\}
\]
intersect, that is, there exist $\tau$ and $\sigma$ such that
\[
x^*(\tau) = x^*(\sigma) = x^*,
\]
with
\[
T_a < \tau < 0 < \sigma < T_r.
\]
Let, for example,
\[
y^*(\tau) < y^*(\sigma).
\]
Then we also require that
\[
Y(x^*, y) < 0, \quad y^*(\tau) < y < y^*(\sigma).
\]
To avoid cumbersome derivations, we additionaly require that the
curves $\Gamma_a$ and $\Gamma_r$ do not self-intersect.
\end{assumption}

\begin{assumption} \label{a:transv}
The intersection is transversal, that is the vectors
$\dot x^*(\tau) $ and $\dot x^*(\sigma) $ are linearly independent.
\end{assumption}

%Assumptions \ref{a:intersect}--\ref{a:transv} are illustrated on Fig.\
%\ref{fig:planes}.
%
%\begin{figure}[htb]
%\begin{center}
%\includegraphics*[width=6cm]{fig/planes.eps}
%\end{center}
%\caption{Solution $w^*(t)$.} \label{fig:planes}
%\end{figure}

Now consider the equation
\begin{equation} \label{e:z-star}
\dot z = Z^*(t, z) = Z(x^*(t), y^*(t), z, 0),
\end{equation}
and denote by $S_{T}$ the shift operator along the solutions of
\eqref{e:z-star} by the time $T$.

The assumptions listed above are (probably?) known and they
guarantee existence of (periodic) canards for the system
\eqref{e:main_0}.

\begin{theorem} \label{t:main}
Let $D \subset \reals^{d}$ be an open bounded set, and let
\[
\gamma(\id - S_{\sigma - \tau}, D) \not= 0.
\]
Then there exist $\eps_0 > 0$ and $\lambda > 0$ such that for any
$\eps < \eps_0$ and any $\hat X, \hat Y$ satisfying
\begin{equation}\label{lambda}
\sup\abs{\hat X(x, y, z, \eps)}, \sup\abs{\hat Y(x, y, z, \eps)} < \lambda
\end{equation}
there exists a periodic solution of \eqref{e:main} that passes
through the set $(B_{\alpha}(x^*), B_{\alpha}(y^*), D)$, with
$\alpha$ going to zero as $\eps_{0},\delta$ go to zero. The minimal
period $T_{\min}$ of this solution approaches $\sigma - \tau$ as
$\eps_{0},\delta \to 0$.
\end{theorem}

\section{Example}
Consider the system
\begin{equation*}
{\arraycolsep=0pt \begin{array}{rl}
\dot x^{(1)} &{} = - a x^{(2)} + y / 3, \\[3pt]
\dot x^{(2)} &{} = x^{(1)} + 1, \\[3pt]
\eps \dot y &{} = x^{(1)} + y^2 + x^{(2)} y.
\end{array} }
\end{equation*}
The curves $\Gamma_a$ and $\Gamma_r$ intersect transversally on the
plane $(x^{(1)}, x^{(2)})$, see Fig.\ \ref{fig:sample-2d}.

\begin{figure}[htb]
\begin{center}
\includegraphics*[width=6cm]{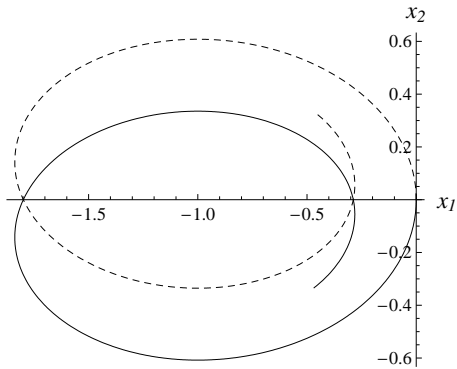}
\end{center}
\caption{Curves $\Gamma_a$ (solid) and $\Gamma_r$ (dashed) for $a = 3$.}
\label{fig:sample-2d}
\end{figure}

%According to Theorem \ref{t:main},
This system has a periodic canard. Figure \ref{fig:sample-2d} graphs
the numerical approximation of this canard, together with the
limiting curve, which consists of $\Gamma_a$, $\Gamma_r$, and a
vertical segment connecting them.

\begin{figure}[htb]
\begin{center}
\includegraphics*[width=6cm]{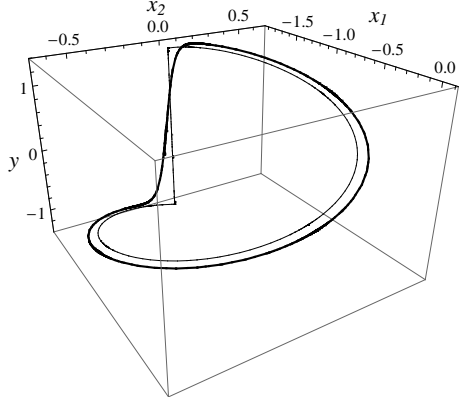}
\end{center}
\caption{Periodic canard for $a = 3$ with $\eps = 0.1$ (thick) and the limiting curve.}
\label{fig:sample-canard}
\end{figure}

According to Theorem \ref{t:main}, {\em the perturbed system
\begin{equation*}
{\arraycolsep=0pt \begin{array}{rl}
\dot x^{(1)} &{} = - a x^{(2)} + y / 3 + \hat X^{(1)}(x, y, \eps) , \\[3pt]
\dot x^{(2)} &{} = x^{(1)} + 1 +\hat X^{(2)}(x, y, \eps), \\[3pt]
\eps \dot y &{} = x^{(1)} + y^2 + x^{(2)} y + \hat Y(x, y, \eps).
\end{array} }
\end{equation*}
has for a small $\eps$ a periodic canard for any $\hat X^{(1)}(x, y,
\eps)$,
 $\hat X^{(2)}(x, y, \eps)$, $\hat Y(x, y, \eps),$
 which are sufficiently small in the uniform norm.}

\section{Chaotic canards} \label{s:main2}

In this section we study the chaotic behavior of canard-type trajectories
of \eqref{e:main}. The method that is used to prove chaoticity
combines the scheme suggested by P.~Zgliczy\'nski \cite{Zgli} with the
method of topological shadowing \cite{Shadowing} and uses the results
obtained in \cite{Homoclinic}. We specifically note that our results
require no computer assisted proofs, in contrast to typical application
of the aforementioned scheme, see \cite{Cox,Zgli2}. For another approach
in studying of chaos in singularly perturbed systems, see \cite{Deng} and
bibliography therein.

\subsection{Definition of chaos}\label{s:main21}

Important attributes of chaotic behavior of a mapping
$f \colon \reals^d \to \reals^d$ include sensitive dependence
on initial conditions, an abundance of periodic trajectories and an
irregular mixing effect describable informally by the existence of a
finite number of disjoint sets which can be visited by trajectories
of $f$ in any prescribed order.

Let ${\mathcal U} = \left\{ U_1, \ldots, U_m \right \}$, $m > 1$,
be a family of disjoint subsets of $\reals^d$ and let us denote
the set of one-sided sequences $\omega = \omega_0, \omega_1, \ldots$
by $\Omega^R_m$. Sequences in $\Omega^R_m$ will be used to prescribe
the order in which sets $U_i$ are to be visited.
For $x \in \bigcup_{i = 1}^m U_i$ we define $I(x)$ to be the number
$i$ satisfying $x \in U_i$.

\begin{definition} \label{def:chaos}

A mapping $f$ is called \emph{${\mathcal U}$-chaotic}, if there exists
a compact $f$-invariant set $S \subset\bigcup_{i} U_{i}$ with
the following properties:

\begin{itemize}
\item[(p1)]
for any $\omega \in \Omega^R_m$ there exists $x \in S$ such that
$f^i (x) \in U_{\omega_i}$ for $i \ge 1$;

\item[(p2)]
for any $p$-periodic sequence $\omega \in \Omega^R_m$
there exists a $p$-periodic point $x \in S$ with
$f^i (x) \in U_{\omega_{i}}$;

\item[(p3)]
for each $\eta > 0$ there exists an uncountable subset $S(\eta)$
of $S$, such that the simultaneous relationships
\[
\limsup_{i \to \infty} \Abs{I(f^i(x)) - I(f^i(y))} \ge 1, \quad
\liminf_{i \to \infty} \Abs{f^i(x) - f^i(y)} < \eta
\]
hold for all $x, y \in S(\eta)$, $x \not= y$.
\end{itemize}

\end{definition}

This definition is a special case of $(\mathcal U, k)$-chaoticity
from \cite{Homoclinic} with $k = 1$.

The above defining properties of chaotic behaviour are similar
to those in the Smale transverse homoclinic trajectory theorem
with an important difference being that we do not require the
existence of an invariant Cantor set. Instead, the definition
includes property (p2), which is usually a corollary of uniqueness,
and (p3), which is a form of sensitivity and irregular mixing as in
the Li--Yorke definition of chaos, with the subset $S(\eta)$
corresponding to the Li--Yorke scrambled subset $S_{0}$.

\subsection{Chaotic behavior} \label{s:main22}
Let us now change Assumption \ref{a:intersect} to the
following stronger assumption, which ensures the existence
of multiple intersections between the curves $\Gamma_a$ and $\Gamma_r$.

\begin{assumption} \label{a:intersect-k}
Let the trajectories $\Gamma_a$ and $\Gamma_r$ intersect
$K \ge 2$ times, that is, there exist
$\tau_i$ and $\sigma_i$, $i = 1, \ldots, K$, such that
\[
x^*(\tau_i) = x^*(\sigma_i) = x^*_i,
\]
with
\[
T_a < \tau_i < 0 < \sigma_i < T_r,
\]
and $\tau_i \not= \tau_j$ for $i \not= j$.
We also require that the curves $\Gamma_a$ and $\Gamma_r$
do not self-intersect, and that
\[
Y(x^*_i, y) < 0, \quad y \in [y^*(\tau_i), y^*(\sigma_i)],
\quad i = 1, \ldots, K.
\]
\end{assumption}

We also change Assumption \ref{a:transv} to the following:

\begin{assumption} \label{a:transv-k}
All the intersections are transversal, $i = 1, \ldots, K$.
\end{assumption}

\begin{theorem} \label{t:chaos-k}
Let $D_i \subset \reals^{d}$ be open, convex and bounded,
and let
\[
S_{\sigma_{j} - \tau_{i}} \bar{D_{i}} \subset D_{j}, \quad i, j = 1, \ldots, K.
\]
Then there exist disjoint sets $\Pi_i \ni x^*(\tau_i)$, $\eps_{0}>0$
and $\lambda > 0$ such that for any $\eps<\eps_{0}$ and any $\hat X,
\hat Y$ satisfying
$$
\sup\abs{\hat X(x, y, z, \eps)},\sup\abs{\hat Y(x, y, z, \eps)} <
\lambda
$$
%Then there exist disjoint sets $\Pi_i \ni x^*(\tau_i)$ such that for
%every sufficiently small $\eps > 0$
the appropriately defined Poincar\'e map ${\mathcal P} \colon
\bigcup_i \Pi_i \times D_i \to \reals^2$ of system \eqref{e:main} is
$\{ \Pi_i \times D_i \}$-chaotic.
\end{theorem}

Let $S \subset \bigcup_i \Pi_i \times D_i$, $i = 1, \ldots, K$, be
the compact ${\mathcal P}$-invariant set from Definition
\ref{def:chaos}; its existence is guaranteed by Theorem
\ref{t:chaos-k}. Denote by ${\mathcal E}_S$ the topological entropy
of the Poincar\'e map ${\mathcal P}$ with respect to the compact set
$S$, see \cite{Katok}, p. 109.

\begin{corollary}
Under conditions of Theorem \ref{t:chaos-k}, for sufficiently small
$\eps$ the following inequality holds:
\[
{\mathcal E}_S \ge \log K.
\]
\end{corollary}

This corollary follows from the $\{ \Pi_i \times D_i \}$-chaoticity
of ${\mathcal P}$ and the definition of topological entropy, see
Proposition 2.1 from \cite{Homoclinic}.

%\section{Generalizations} \label{s:g}
%Theorems hold if in an appropriate coordinate system $(p, q, h)$ the
%slow surface $S_0$ near the origin takes the form
%\begin{equation}\label{mm1}
%h^w = -p D(p, q), \quad p \le 0,
%\end{equation}
%where $w>0$ and $D(p, q)$ is a smooth function satisfying $D(0, 0) >
%0$.

\section{Proofs} \label{p:g}

\subsection{Proof of Lemma \ref{existL}}
The equation \eqref{e:attr} has a singularity at the origin,
as demonstrated by \eqref{e:attr2}. Thus, we only need to show
the existence and uniqueness of the solution at the origin.
To do this, we eliminate the third equation from \eqref{e:attr} by
rewriting it in a specially selected curvilinear coordinate system
in a vicinity of the origin, and by proving the one-sided Lipschitz
conditions for the transformed system.

The curvilinear coordinate system $(p, q, h)$ is introduced in
the following way: we keep the $x^{(1)}$ and $y$ axes from the
$(x^{(1)}, x^{(2)}, y)$ coordinate system, so that they
become the $p$ and $r$ axes correspondingly, and direct the
$q$ axis along the turning line $L$ on the slow surface $S_0$.
In the new coordinates system \eqref{e:main_0} takes the form
\begin{equation}
{\arraycolsep=0pt \begin{array}{rl}
\dot p &{}= P(p, q, h, \eps), \\[3pt]
\dot q &{}= Q(p, q, h, \eps), \\[3pt]
\eps \dot h &{}= H(p, q, h, \eps).
\end{array} }
\end{equation}
This representation is valid in a sufficiently small vicinity of
the origin, which we denote by
$\Omega_0 = \{ (p, q, h) : \abs{p}, \abs{q}, \abs{h} < \delta_0 \}$.
Recall that the turning line of the surface $S_0$ is
described by equations
\begin{equation*}
{\arraycolsep=0pt \begin{array}{rl}
Y(x_1, x_2, y, 0) = 0, \\[3pt]
Y'_y(x_1, x_2, y, 0) = 0.
\end{array} }
\end{equation*}
Due to the choice of $q$, these equations take the form
\[
p = 0, \quad h = 0
\]
in the $(p, q, h)$ coordinates, implying
\begin{equation} \label{koordE}
H(0, q, 0, 0) = H'_h(0, q, 0, 0) = 0
\end{equation}
in a sufficiently small vicinity of the origin.
Let us now calculate the tangent to the turning line at the origin.
Due to \eqref{grE} and \eqref{geE}, $Y(x_1, x_2, y, 0)$ admits
the following representation at zero:
\[
Y(x_1, x_2, y, 0) = \xi x_1 + \zeta y^2 / 2 + \phi y x_2 + a x_2^2
+ b x_1^2 + c x_1 x_2 + d x_1 y + o(\norm{(x_1, x_2, y)}^2),
\]
where $\xi = Y'_{x^{(1)}}(0) > 0$, $\zeta = Y''_{yy}(0) > 0$,
$\phi = Y''_{x^{(2)} y}(0)$. Thus, the equations of the turning line
can be approximated at zero as
\begin{equation*}
{\arraycolsep=0pt \begin{array}{rl}
\xi x_1 + \zeta y^2 / 2 + \phi y x_2 + a x_2^2 + b x_1^2 + c x_1 x_2 + d x_1 y = 0, \\[3pt]
\zeta y + \phi x_2 + d x_1 = 0.
\end{array} }
\end{equation*}
From the first equation we have that $x_1$ has an order of square, thus
it can be eliminated from the second equation. Therefore, the tangent
vector at the origin is
\begin{equation} \label{e:pqr-tangent}
(0, \zeta, -\phi).
\end{equation}
Inequalities \eqref{geE} imply that this vector, and therefore the
$q$ axis, forms a sharp angle with the the vector
$\dot w^*(0) = (0, X^{(2)}(0, 0, 0, 0), 0)$, thus
\eqref{e:x1at0} and \eqref{geE} become
\begin{equation} \label{e:pqr-cond}
P(0, 0, 0, 0) = 0, \qquad Q(0, 0, 0, 0) > 0, \qquad H''_{hh}(0, 0, 0, 0) > 0,
\end{equation}
and the Assumption \ref{ineA} translates into
\begin{equation} \label{e:pqr-a1}
P'_q(0, 0, 0, 0) < 0, \qquad P'_h(0, 0, 0, 0) > 0.
\end{equation}
In the new coordinate system $(p, q, h)$ the
%slow surface $S_0$ near the origin takes the form
%\begin{equation}\label{mm1}
%h^2 = -p D(p, q), \quad p \le 0,
%\end{equation}
%where $D(p, q)$ is a smooth function satisfying $D(0, 0) > 0$.
%Thus,
The attractive part $S_a$ of $S_0$ satisfies the equation
\begin{equation} \label{mma1}
h = - \sqrt{-p D_{a}(p, q)}, \quad p < 0,
\end{equation}
and the repulsive part $S_r$ satisfies the equation
\begin{equation} \label{mmr1}
h = \sqrt{-p D_{r}(p, q)}, \quad p < 0,
\end{equation}
where $D_{a}(p, q),$ $D_{r}(p, q),$ are smooth functions satisfying
$D_{a}(0, 0), D_{r}(0, 0)> 0$.

 Substituting $r$ with \eqref{mma1} and \eqref{mmr1} into
\eqref{e:attr}, we obtain a system of two equations describing the
half-trajectory of \eqref{e:attr} that lives on the attractive part
of the slow surface:
\begin{equation} \label{e:attr-pq}
{\arraycolsep=0pt \begin{array}{rl}
\dot p &{} = P_a(p, q) = P(p, q, -\sqrt{-p D_{a}(p, q)}, 0), \\[3pt]
\dot q &{} = Q_a(p, q) = Q(p, q, -\sqrt{-p D_{a}(p, q)}, 0),
\end{array} }
\end{equation}
and a set of equations describing the half-trajectory of
\eqref{e:attr} that lives on the repulsive part of the slow
surface:
\begin{equation} \label{e:repul-pq}
{\arraycolsep=0pt \begin{array}{rl}
\dot p &{} = P_r(p, q) = P(p, q, \sqrt{-p D_{r}(p, q)}, 0), \\[3pt]
\dot q &{} = Q_r(p, q) = Q(p, q, \sqrt{-p D_{r}(p, q)}, 0),
\end{array} }
\end{equation}
with the initial condition $p(0) = q(0) = 0$ being the same for
both \eqref{e:attr-pq} and \eqref{e:repul-pq}. The existence of
solutions of \eqref{e:attr-pq} and \eqref{e:repul-pq} follows
from the continuity of the right-hand side. To establish
the uniqueness of solution of \eqref{e:attr-pq} in negative time,
we divide the first equation from \eqref{e:attr-pq} by the second,
obtaining
\[
\frac{d p}{d q} = P_a(p, q) / Q_a(p, q).
\]
Then, taking into account that $Q_a(0, 0) > 0$, we prove the one-sided
Lipschitz condition for $P_a(p, q)$ in the $p$ variable:
\[
(P_a (p_1, q) - P_a (p_2, q))(p_1 - p_2) \ge - L_a (p_1 - p_2)^2,
\]
where $-\eps_L \le p_1, p_2 \le 0$, $\abs{q} \le \eps_L$, and $L_a
\ge 0$. This follows from the elementary estimate
\[
\frac{\partial}{\partial p} P_a(p, q) > L,
\]
which holds in a small vicinity of zero for an appropriate $L<0$.
%\textbf{(Insert proof of one-sided Lipschitz)}
This proves the uniqueness of the solution $p(q)$, and uniqueness of
$p(t)$ and $q(t)$ follows.
%\textbf{(Insert proof of uniqueness of
%$p(t)$ and $q(t)$)}

Uniqueness of the solution of \eqref{e:repul-pq} in positive time
is proved in the same way.

\subsection{Proof of Theorem \ref{t:main}}

Consider the coordinate system $(p, q, h)$ introduced in a vicinity
of zero in Lemma \ref{existL}. We extend this system to the whole
space $\reals^3$ by aligning the $q$ axis along its tangent
vector at zero \eqref{e:pqr-tangent} outside of this vicinity,
and connecting to the curvilinear part in a differentiable way.
Thus, we get a global almost linear coordinate change, coinciding
with the curvilinear change at the origin.

From this moment, we will be working with this new coordinate
system, so \eqref{e:main_0} takes the form
\begin{equation} \label{e:main-pq}
{\arraycolsep=0pt \begin{array}{rl}
\dot p &{}= P(p, q, h, \eps), \\[3pt]
\dot q &{}= Q(p, q, h, \eps), \\[3pt]
\eps \dot h &{}= H(p, q, h, \eps),
\end{array} }
\end{equation}
and \eqref{e:attr} takes the form
\begin{equation} \label{e:attr3-pq}
{\arraycolsep=0pt \begin{array}{rl}
\dot p &{} = P (p, q, h, 0), \\[3pt]
\dot q &{} = Q (p, q, h, 0), \\[3pt]
(p, q, h) &{} \in S_0.
\end{array} }
\end{equation}
Because the $y$ axis becomes the $h$ axis in the new coordinates,
the relationships \eqref{e:l1gtz} and \eqref{e:l1ltz} from Lemma \ref{existL}
also hold for \eqref{e:main-pq}:
\begin{alignat*}{2}
H_{h} (w^*(t)) &{}> 0, &\quad 0 < {}& t < T_{r}, \\
H_{h} (w^*(t)) &{}< 0, &\quad T_{a} < {}& t < 0.
\end{alignat*}
Hence, the attractive and repulsive slow surfaces
allow the following representation in a certain
vicinities of the curves $\Gamma_a$ and $\Gamma_r$
correspondingly:
\begin{equation} \label{e:ha-hr}
h = h_a(p, q), \qquad h = h_r(p, q),
\end{equation}
where the functions $h_a$ and $h_r$ are smooth.
Denote this vicinities of $\Gamma_a$ and $\Gamma_r$
by $\Omega_a$ and $\Omega_r$ correspondingly.
In the vicinity $\Omega_0$ of the origin these parametrizations
take the form \eqref{mma1} and \eqref{mmr1}.

%To simplify the following text, we redefine the function
%$H(p, q, h, \eps)$ outside of the vicinities of
%$\Gamma_a$ and $\Gamma_r$ in such a way, that
%$h_a(p, q)$ and $h_r(p, q)$

Substituting $h_a$ and $h_r$ into \eqref{e:attr3-pq}, we obtain
autonomous differentian equations which describe the dynamics
of the $(p, q)$-component of the attractive slow solutions
of the main system \eqref{e:main-pq} for negative $t$:
\begin{equation} \label{e:attr_n}
{\arraycolsep=0pt \begin{array}{rl}
\dot p &{} = P_a (p, q) = P(p, q, h_a(p, q), 0), \\[3pt]
\dot q &{} = Q_a (p, q) = Q(p, q, h_a(p, q), 0),
\end{array} }
\end{equation}
and of the repulsive slow solutions for positive $t$:
\begin{equation} \label{e:repul_n}
{\arraycolsep=0pt \begin{array}{rl}
\dot p &{} = P_r (p, q) = P(p, q, h_r(p, q), 0), \\[3pt]
\dot q &{} = Q_r (p, q) = Q(p, q, h_r(p, q), 0),
\end{array} }
\end{equation}
These representations are valid in a vicinity of the curves $\Gamma_a$
and $\Gamma_r$ correspondingly. In the vicinity of zero
\eqref{e:attr_n} becomes \eqref{e:attr-pq} and \eqref{e:repul_n}
becomes \eqref{e:repul-pq}.

Consider now a small vicinity of the intersection point $(p_*, q_*)$
of the curves $\Gamma_a$ and $\Gamma_r$, which existence is
guaranteed by Assumption \ref{a:intersect}. Let $(p_0, q_0)$
be a point in this vicinity. Denote by $w_a(t, t_0, p_0, q_0) =
(p_a(t, t_0, p_0, q_0), q_a(t, t_0, p_0, q_0))$ the solution
of \eqref{e:attr_n} with the initial condition $p(t_0) = p_0$, $q(t_0) = q_0$,
and by $w_r(t, t_0, p_0, q_0) = (p_r(t), q_r(t))$
the solution of \eqref{e:repul_n} with the same initial condition
$p(t_0) = p_0$, $q(t_0) = q_0$.

Denote
\[
A = P_a(p^*, q^*) Q_r(p^*, q^*) - P_r(p^*, q^*) Q_a(p^*, q^*).
\]
Since the intersection between $\Gamma_a$ and $\Gamma_r$ is transversal
according to Assumption \ref{a:transv}, $A \not= 0$, and the numbers
$u(p, q)$ and $v(p, q)$ may be defined in this vicinity by
\[
w_a(u(p, q), 0, p, q) \in \Gamma_r, \quad
w_r(v(p, q), 0, p, q) \in \Gamma_a.
\]

\begin{figure}[htb]
\begin{center}
\includegraphics*[width=7cm]{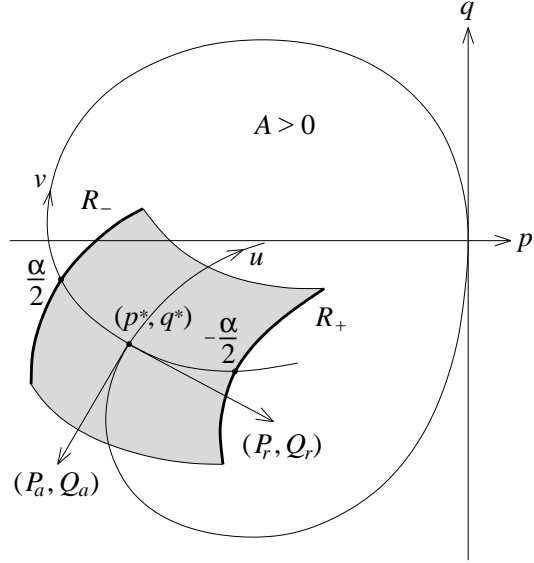}
\end{center}
\caption{The set $\Pi(\alpha)$ on the plane $(p, q)$.}
\label{fig:pi-rect}
\end{figure}

Using the coordinates $u(p, q)$ and $v(p, q)$ we introduce
for a sufficiently small $\alpha > 0$ a `parallelogram' set
$\Pi(\alpha)$, illustrated on Fig.\ \ref{fig:pi-rect}:
\begin{equation} \label{e:par}
\Pi(\alpha) = \{ (p, q) \colon \abs{u(p, q)}, \abs{v(p, q)} < \alpha / 2 \}.
\end{equation}
Also introduce the notation for the two `sides' of this parallelogram:
\begin{alignat*}{1}
R_- &{} = \{ (p, q) \colon v(p, q) = +\alpha /2 \sgn A, \abs{u(p, q)} \le \alpha / 2 \}, \\
R_+ &{} = \{ (p, q) \colon v(p, q) = -\alpha /2 \sgn A, \abs{u(p, q)} \le \alpha / 2 \}.
\end{alignat*}

The solutions $w_a(t, p_0, q_0) = w_a(t, \tau, p_0, q_0)$ of
the system \eqref{e:attr_n} with the initial condition
$(p(\tau), q(\tau)) = (p_0, q_0) \in \Pi$ have several
important properties described below. To formulate them,
we will need some auxiliary definitions.

Define the set
\[
E = \{ (p, q) \colon (p = 0 \wedge q \le 0) \vee (p \le 0 \wedge q = 0) \},
\]
which is the union of the left half of the horizontal $p$ coordinate
axis and the bottom half of the vertical $q$ coordinate axis.
Consider a set ${\mathcal T}(p_0, q_0)$ of all time moments when the
solution $w_a(t, p_0, q_0)$ intersects the set $E$:
\[
{\mathcal T}(p_0, q_0) = \{ t \colon w_a(t, p_0, q_0) \in E \}.
\]

\begin{figure}[htb]
\begin{center}
\includegraphics*[width=4cm]{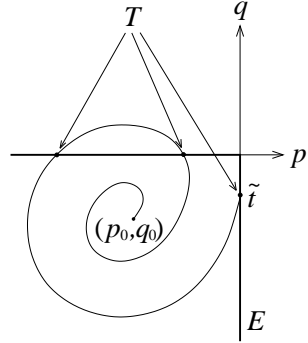}
\end{center}
\caption{The set ${\mathcal T}$ and the moment $\tilde t$.}
\label{fig:t-tilde}
\end{figure}

\begin{definition}
Let $\tilde t(p_0, q_0)$ be the moment from the set ${\mathcal T}(p_0, q_0)$
closest to zero, see Fig.\ \ref{fig:t-tilde}:
\[
\tilde t = \argmin_{t \in {\mathcal T}(p_0, q_0)} \abs{t}.
\]
Below we sometimes omit the point $(p_0, q_0)$ and write simply $\tilde{t}$,
when the arguments can be uniquely identified from the context.
\end{definition}

\begin{definition}
If $\tilde{t}(p_0, q_0)$ exists, and $p_a(\tilde t, p_0, q_0) = 0$
and $q_a(\tilde t, p_0, q_0) < 0$, then the solution $w_a(t, p_0, q_0)$
together with the point $(p_0, q_0)$ are called \emph{destabilizing}.
If $p_a(\tilde t, p_0, q_0) < 0$ and $q_a(\tilde t, p_0, q_0) = 0$, then
the solution $w_a(t, p_0, q_0)$ and the point $(p_0, q_0)$
are called \emph{stabilizing}.
\end{definition}

This definition emphasizes the fact that the solution
$w_\eps(t, x_0, y_0)$ of the main equation \eqref{e:main}
with a stabilizing initial condition will stay close
to the attractive half-plane $P_a$ for some time after $t > 0$,
if $\eps$ is sufficiently small; if initial condition
is destabilizing, then $h_\eps(t)$ will rapidly increase
after $t > 0$. The proof of this fact will be the subject
of several propositions, all leading to Proposition \ref{p:case3}.

\begin{statement} \label{s:1}
If $\alpha$ is sufficiently small, then the moment $\tilde t$
is defined for all $(p_0, q_0) \in \Pi$ and is continuous
on $\Pi$ with respect to $(p_0, q_0)$.
\end{statement}

\begin{statement} \label{s:2}
If $\alpha$ is sufficiently small, then a point $(p_0, q_0) \in \Pi$ is
destabilizing if $A \cdot v(p_0, q_0) > 0$, and stabilizing if $A \cdot v(p_0,
q_0) < 0$. In particular, $R_-$ is destabilizing and $R_+$ is stabilizing.
\end{statement}

\begin{statement} \label{s:3}
Let $\alpha$ be suficiently small. If a point $(p_0, q_0)$ is
stabilizing, then $H_{h} (w_a(t, p_0, q_0)) < 0$ for $\tau \le t \le \delta(\alpha)$
where $\delta(\alpha) > \max_\Pi \{ \tilde t (p_0, q_0) \} > 0$
depends only on $\alpha$.
\end{statement}

Statement \ref{s:1} follows from the continuous dependence of $w_a$ on $(p_0,
q_0)$ and from the fact that $w_a^*$ intersects the line $q = 0$
transversally.

To prove Statements \ref{s:2} and \ref{s:3}, we can consider the projection
$d(t)$ of the difference between the trajectories $w_a(t, p_0, q_0)$ and
$w^*(t)$ onto the normal vector for $w^*(t)$, $t \ge \tau$:
\[
d(t) = (p_a(t) - p^*(t)) Q_a(p^*(t), q^*(t)) -
       (q_a(t) - q^*(t)) P_a(p^*(t), q^*(t)).
\]
Then the variation of $d(t)$ satisfies the following initial value problem:
\[
\dot r(t) = (\tfrac{\partial}{\partial p} P_a(p^*(t), q^*(t)) +
\tfrac{\partial}{\partial q} Q_a(p^*(t), q^*(t))) r(t), \quad r(\tau) = 1,
\]
and the following equality holds for small $\Delta p_0 = p_0 - p_*$,
$\Delta q_0 = q_0 - q^*$, $d_0 = \Delta p_0 Q_a(p^*, q^*) -
\Delta q_0 P_a(p^*, q^*)$:
\[
d(t) = r(t) d_0 + o(\Delta p_0) + o(\Delta q_0).
\]
Therefore, if $d_0 > 0$, and $\Delta p_0$ and $\Delta q_0$ are small,
then $d(t) > 0$ for $\tau \le t$. Assumption \ref{a:transv}
provides that $d_0 > 0$ for $(p_0, q_0) \in R_-$, and
$d_0 < 0$ for $(p_0, q_0) \in R_+$.

Denote by
\begin{equation} \label{e:sol-eps}
w_\eps(t, p_0, q_0) = (p_\eps(t, p_0, q_0), q_\eps(t, p_0, q_0), h_\eps(t, p_0, q_0))
\end{equation}
the solution of the system \eqref{e:main-pq} with the initial condition
\[
p(\tau) = p_0, \quad q(\tau) = q_0, \quad h(\tau) = h_0 = (h^*(\tau) + h^*(\sigma)) / 2,
\]
with $(p_0, q_0) \in \Pi(\alpha)$.

\begin{definition}
We say that $a$ is $\eps$-close to $b$, if $a \to b$ as $\eps \to 0$.
\end{definition}

In our assumptions, the solution $w_\eps(t, p_0, q_0)$ first rapidly
approaches the attractive half-plane $S_a$,
and then follows $S_a$ until $t \approx 0$. The possible subsequent
behaviors of the solution are the following:

\begin{enumerate}
\item The solution stays close to the attractive part of the slow surface
for $t \le \delta$, where $\delta > 0$.

\item The solution begins to rapidly increase in the $h$ direction
around $t \approx 0$.

\item The solution follows the repulsive half-plane $S_r$
until a moment $t^*$. Moreover, for positive $t < t^*$ it follows
the curve $\Gamma_r$. After the moment $t^*$
the solution may begin to rapidly increase in the $h$ direction,
or it may fall down back to the attractive part of the slow surface.
\end{enumerate}

The propositions below provide a more rigorous description of the
qualitative description above.

\begin{proposition} \label{p:heps-ha}
Let $\eps$ be sufficiently small, and let $(w_a(t, p_0, q_0), h_a(w_a(t))) \in S_a$
for $\tau \le t \le T$ (this means that the corresponding solution
of \eqref{e:attr3-pq} does not cross the turning line $L$). Then
the difference $h_\eps(t) - h_a(p_\eps(t), q_\eps(t))$ becomes $\eps$-small
at the moment $t_1$ which is $\eps$-close to $\tau$, and stays
$\eps$-small for $t_1 \le t \le T$.
\end{proposition}

\begin{proof}
Lemma \ref{existL} and the choice of $\alpha$ imply that the inequality
\[
-C_2 \le H'_h(p, q, h, \eps) \le -C_1 < 0, \qquad C_1, C_2 > 0,
\]
holds in a vicinity $\Omega$ of the curve
$(w_a(t, p_0, q_0), h_a(w_a(t)))$, $\tau \le t \le T$. Suppose also
without loss of generality that all the functions in the right-hand side
of \eqref{e:main-pq} are bounded by a constant $M > 0$ along with all the
first derivatives.

By virtue of Assumption \ref{a:intersect}, $H(p, q, h, \eps) \le H_0 < 0$
in a vicinity of the vertical line
$p = p^*, q = q^*, h^*(\tau) + \Delta h \le h \le h_0$, where $\Delta h$
is sufficiently small. Thus, the vertical speed of the solution $w_\eps(t)$
goes to infinity as $\eps \to 0$, and therefore the solution reaches the
bottom end of this vicinity in an $\eps$-small time (linear
in $\eps$). The changes in $p_\eps$ and $q_\eps$ are $\eps$-small after this
time, thus the solution exits through the bottom part of this vicinity,
entering $\Omega$.

Denote $\phi(t) = h_\eps(t) - h_a(p_\eps(t), q_\eps(t))$.
After the solution entered $\Omega$, we have for $\phi \ge 0$:
\begin{alignat*}{1}
\dot\phi(t) &{}= \dot h_\eps(t) - \dot h_a(p_\eps(t), q_\eps(t)) =
\frac{1}{\eps} H(p, q, h, \eps) - \frac{\partial h_a}{\partial p} \dot p_\eps(t)
- \frac{\partial h_a}{\partial q} \dot q_\eps(t) \\
&{}= \frac{1}{\eps} H(p, q, h, \eps) + \frac{H'_p}{H'_h} P(p, q, h, \eps)
+ \frac{H'_q}{H'_h} Q(p, q, h, \eps) \\
&{}\le \frac{1}{\eps} H(p, q, \phi + h_a(p, q), \eps) + \frac{M^2}{C_1}
= \frac{1}{\eps} (H'_h \cdot \phi(t) + H'_\eps \cdot \eps) + \frac{M^2}{C_1} \\
&{}\le -\frac{C_1}{\eps} \phi(t) + M + \frac{M^2}{C_1}.
\end{alignat*}
By the theorem on differential inequalities,
\[
\phi(t) \le \bar \phi(t) = (\phi_0 - \eps C) e^{\frac{-C_1 t}{\eps}} + \eps C,
\]
where $C$ depends on $C_1$ and $M$. Calculate the time moment
when $\phi(t)$ becomes equal to $2 \eps C$:
\begin{gather*}
(\phi_0 - \eps C) e^{\frac{-C_1 t}{\eps}} = \eps C, \\
t = \frac{\eps}{C_1} \ln \frac{\phi_0 - \eps C}{\eps C} \le
\frac{\eps}{C_1} \ln \frac{\Delta h^*}{\eps C},
\end{gather*}
where $\Delta h^* = h^*(\sigma) - h^*(\tau)$.
Thus, $\phi(t)$ gets $\eps$-close to zero from above after
an $\eps$-small time interval $[\tau, t_1]$.

We obtain the lower bound for $\phi(t)$ in the same way:
\begin{alignat*}{1}
\dot\phi(t) &{}\ge \frac{1}{\eps} H(p, q, \phi + h_a(p, q), \eps) - \frac{M^2}{C_1}
= \frac{1}{\eps} (H'_h \cdot \phi(t) + H'_\eps \cdot \eps) - \frac{M^2}{C_1} \\
&{}\ge -\frac{C_2}{\eps} \phi(t) - M - \frac{M^2}{C_1},
\end{alignat*}
where $\phi(t) > 0$. Therefore
\[
\phi(t) \ge (\phi_0 + \eps C) e^{\frac{-C_1 t}{\eps}} - \eps C.
\]
Similar equations can be written for the case $\phi(t) < 0$
with $C_2$ instead of $C_1$. In any case, $\phi(t)$ remains
$\eps$-close to zero for $t_1 \le t \le T$.
\end{proof}

\begin{proposition} \label{p:inv-sets}
Consider the set $U_1(\beta) = \{ (p, q, h) : h > h_a(p, q) - \beta \}$
defined in the vicinity $\Omega_a$ of the curve $\Gamma_*$ for $t < 0$
by virtue of \eqref{e:ha-hr}, and in the vicinity $\Omega_0$ of the origin
for $p \le 0$ by virtue of \eqref{mma1}. Also introduce the set
$U_2(\beta) = \{ (p, q, h) : h > p - \beta \}$ in $\Omega_0$
for $p > 0$, and let $U_a = U_1 \cup U_2$. Denote
\[
\gamma(\eps) = \inf \{ \beta \colon w_\eps(t, p_0, q_0) \in U_a(\beta)
\text{ for any } (p_0, q_0) \in \Pi \}.
\]
Then $\gamma(\eps)$ is $\eps$-small.
\end{proposition}

\begin{proof}
Suppose that $\gamma(\eps)$ is not $\eps$-small. Then there exist
$\gamma_0 > 0$, $\eps_n \to 0$, $(p_n, q_n)$, and $t_n$, such that
$w_\eps(t_n, p_n, q_n) \not\in U_a(\gamma_0)$. Denote the limit point
of the sequence $w_\eps(t_n, p_n, q_n)$ by $w_0 \not\in U_a(\gamma_0)$.
Suppose that $w_0 \not\in \Omega_0$, or $w_0 \in \Omega_0$ and $p_0 \le 0$.
Consider the scalar product $\langle \dot w_\eps, \mathop{\mathrm{grad}} S_a \rangle$
at the point $w_0$:
\[
C_1 = \langle \dot w_\eps, \mathop{\mathrm{grad}} S_a \rangle =
-P H'_p - Q H'_q - \frac{1}{\eps} H H'_h.
\]
Note that $H(w_0) > 0$, $H'_h(w_0) < 0$, and the first and second terms are
bounded. Therefore we can select a sufficiently
small $\eps > 0$ such that this scalar product is positive, which means
that the velocity vector $\dot w_\eps$ is directed inside the set
$U_a(\gamma_0)$. Due to the choice of $w_0$ we can select $n$ such that
$\eps_n < \eps$, and $w_\eps(t_n, p_n, q_n)$ is sufficiently close
to $w_0$, therefore this scalar product should be close to $C_1$ by
virtue of continuity of the right hand side of \eqref{e:main-pq}, and thus
positive. However, the velocity vector $\dot w_\eps$ is not directed
inside the set $U_a(\gamma_0)$, implying that the scalar product must
be non-positive, contradicting positiveness of $C_1$.

In the case $w_0 \in \Omega_0$, $p_0 > 0$, the scaler product
$\langle \dot w_\eps, \mathop{\mathrm{grad}} \partial U_2(\gamma_0) \rangle$
takes the form
\[
C_2 = \langle \dot w_\eps, \mathop{\mathrm{grad}} \partial U_2(\gamma_0) \rangle =
-P + \frac{1}{\eps} H.
\]
In this case $H(w_0) > 0$, and the same reasoning applies.
This contradiction proves that $\gamma(\eps)$ is $\eps$-small.
\end{proof}

This proposition shows that there exists an invariant set
$\eps$-close to $S_a$, that the trajectories $w_a(t, p_0, q_0)$
of \eqref{e:main-pq} with $(p_0, q_0) \in \Pi$ do not intersect.
By applying the same reasoning to solutions of \eqref{e:main-pq}
in backward time, we can obtain the existence of an invariant set $U_r$
$\eps$-close to $S_r$, that the trajectories $w_a(t, p_0, q_0)$
which return to the vicinity of the point $(p^*, q^*, h_0)$
also do not intersect:
\[
U_r = \{ (p, q, h) : h > h_a(p, q) - \gamma(\eps) \} \cup
      \{ (p, q, h) : p > 0 \wedge h > p - \gamma(\eps) \}.
\]

\begin{proposition} \label{p:stable}
If $(x_0, y_0)$ is stabilizing and $\eps$ is sufficiently small,
then $p_\eps(t, p_0, q_0)$ is $\eps$-close to $p_a(t, p_0, q_0)$, and
$q_\eps(t, p_0, q_0)$ is $\eps$-close to $q_a(t, p_0, q_0)$
for all $\tau \le t \le \delta(\alpha)$.
\end{proposition}

\begin{proof}
According to Statement \ref{s:3}, $H_{h} (w_a(t, p_0, q_0)) < H_0 < 0$
for $\tau \le t \le \delta(\alpha)$. Proposition \ref{p:heps-ha}
implies in this case that $h_\eps(t) - h_a(p_\eps(t), q_\eps(t))$
is $\eps$-small for $t_1 \le t \le \delta(\alpha)$, where $t_1$ is
$\eps$-close to $\tau$. Therefore, first we have
\begin{alignat*}{1}
\abs{p_\eps(t_1) - p_a(t_1)} &{}\le \abs{p_\eps(t_1) - p_0} + \abs{p_a(t_1) - p_0}
    \le 2 M (t_1 - \tau), \\
\abs{q_\eps(t_1) - q_a(t_1)} &{}\le 2 P (t_1 - \tau),
\end{alignat*}
and the statement of the poposition holds for $\tau \le t \le t_1$.

Denote $\psi(t) = \abs{p_\eps(t) - p_a(t)} + \abs{q_\eps(t) - q_a(t)}$, $t \ge t_1$.
Using \eqref{e:main-pq} and \eqref{e:attr-pq}, we obtain
\begin{alignat*}{1}
D_R \abs{p_\eps - p_a} &{}= \abs{P(p_\eps, q_\eps, h_a(p_\eps, q_\eps), \eps)
    - P(p_a, q_a, h_a(p_a, q_a), 0)} + P'_h \gamma_1(\eps) \\
    &{}= (P'_p - P'_h H'_p / H'_h) \abs{p_\eps - p_a} + (P'_q - P'_h H'_q / H'_h) \abs{q_\eps - q_a} +
    P'_\eps \eps + P'_h \gamma_1(\eps) \\
    &{}\le C (\abs{p_\eps - p_a} + \abs{q_\eps - q_a} + \gamma_1(\eps)), \\
D_R \abs{q_\eps - q_a}
    &{}\le C (\abs{p_\eps - p_a} + \abs{q_\eps - q_a} + \gamma_1(\eps)),
\end{alignat*}
where $D_R$ denotes the right derivative, $\gamma_1(\eps)$ is $\eps$-small,
and $C$ is a sufficiently large constant. Therefore
\[
D_R \psi(t) \le C_1 \psi(t) + \gamma_2(\eps), \quad \psi(t_1) \le \gamma_2(\eps),
\]
and by the theorem on differential inequalities,
\[
\psi(t) \le \gamma_3(\eps) e^{C_1 (t - t_1)} - \gamma_3(\eps),
\]
which is $\eps$-small on a finite time interval $t_1 \le t \le \delta(\alpha)$.
\end{proof}

\begin{proposition} \label{p:unstable}
Let $(x_0, y_0)$ be destabilizing and $\eps$ be sufficiently small.
Consider a small vicinity $\tilde\Omega$ of the point
$(0, \tilde q, 0) = (p_a(\tilde t), q_a(\tilde t), 0)$, and let $\hat t$
be the time when the solution $w_a(t)$ is inside this vicinity.
Then $p_\eps(t, p_0, q_0)$ is $\eps$-close to $p_a(t, p_0, q_0)$, and
$q_\eps(t, p_0, q_0)$ is $\eps$-close to $q_a(t, p_0, q_0)$
for all $\tau \le t \le \hat t$, and $w_\eps(t)$ exits
the invariant set $U_r$ when it leaves $\tilde\Omega$.
\end{proposition}

\begin{proof}
By noting that $H_{h} (w_a(t, p_0, q_0)) < H_0 < 0$ for $\tau \le t \le \hat t$
and repeating the same steps as in Proposition \ref{p:stable}, we
obtain that $p_\eps(t, p_0, q_0)$ is $\eps$-close to $p_a(t, p_0, q_0)$, and
$q_\eps(t, p_0, q_0)$ is $\eps$-close to $q_a(t, p_0, q_0)$
for all $\tau \le t \le \hat t$. This means that $w_\eps(t)$
is inside $\tilde\Omega$ at the moment $\hat t$.

The vicinity $\Omega_0$ should be selected small enough that $Q \ge Q_0 > 0$
in this vicinity, and $\alpha$ and $\tilde\Omega$ should be small enough
such that $\tilde\Omega \subset \Omega_0$. \eqref{e:pqr-a1} together
with $\tilde q < 0$ imply that $P \ge P_0 > 0$ in $\tilde\Omega$.
Let
\[
\tilde\Omega = \{ (p, q, h) : \abs{p} < \tilde\delta_p,
\abs{q - \tilde q} < \tilde\delta_q, \abs{h} < \tilde\delta_h \},
\]
with sufficiently small $\tilde\delta_p$, $\tilde\delta_q$, $\tilde\delta_h$,
such that $\tilde\delta_p < \frac{P_0}{2 M} \tilde\delta_q$, and the bottom
side of $\tilde\Omega$ with $h = -\tilde\delta_h$ is outside $U_a$, and the
top side $h = \tilde\delta_h$ is outside $U_r$.

Consider the time moment $\bar t$ when $w_\eps(t)$ leaves
this vicinity. It can do this in three possible ways:
\begin{enumerate}
\item $p_\eps(\bar t) = \tilde\delta_p$. Then $h_\eps(\bar t) > -\gamma(\eps)$,
thus $w_\eps(t) \not\in U_r$.
\item $q_\eps(\bar t) = \tilde\delta_q$. Due to the choice of $\delta_p$,
$p_\eps(\bar t) > \tilde\delta_p$, thus this case is not possible, and case 1
actually takes place.
\item $h_\eps(\bar t) = \tilde\delta_h$. In this case $w_\eps(\bar t) \not\in U_r$
by choice of $\tilde\delta_h$.
\end{enumerate}
In all cases, $w_\eps$ leaves $U_r$.
\end{proof}

Propositions \ref{p:stable} and \ref{p:unstable} are both valid
for sufficiently small $\eps$, that is, for $\eps < \eps_0(p_0, q_0)$.
Due to continuous dependence of the solutions on initial data,
$\eps_0(p_0, q_0)$ can be selected in such a way that these
propositions will be valid for $\eps < \eps_0$ in some vicinity
of the point $(p_0, q_0)$. The next proposition shows, that
Case 3 above is only possible when the initial condition
$(p_0, q_0) \in \Pi$ is $\eps$-close to the curve $\Gamma_a$.

\begin{proposition} \label{p:case3}
For every $\eps$ define a set of points from $\Pi$ such that
Propositions \ref{p:stable}-\ref{p:unstable} hold for these
points with the selected $\eps$:
\[
\Delta(\eps) = \{ (p_0, q_0) \in \Pi \colon \eps_0(p_0, q_0) \ge \eps \}.
\]
Denote
\[
\gamma(\eps) = \sup \{ \abs{u(p_0, q_0)} \colon (p_0, q_0) \not\in \Delta(\eps) \}.
\]
Then $\gamma(\eps)$ is $\eps$-small. In other words, if a trajectory
$w_\eps(t, p_0, q_0)$ does not fulfill Propositions \ref{p:stable}-\ref{p:unstable},
then $v(p_0, q_0)$ is $\eps$-small.
\end{proposition}

\begin{proof}
Suppose that $\gamma(\eps)$ is not $\eps$-small. Then there exist
$\gamma_0$, $\eps_n \to 0$, $(p_n, q_n) \in \Pi$, such that
$(p_n, q_n) \not\in \Delta(\eps_n)$, and $\abs{v(p_n, q_n)} > \gamma_0$.
Consider a limit point $(\hat p, \hat q)$ of the sequence $(p_n, q_n)$,
$\abs{v(\hat p, \hat q)} \ge \gamma_0$. The latter inequality means that
$(\hat p, \hat q)$ is either stabilizing or destabilizing, thus
$\eps_0(\hat p, \hat q)$ is defined and positive. Select $\eps_n$
such that $\eps_n < \eps_0$, and $(p_n, q_n)$ is in a vicinity of
the point $(\hat p, \hat q)$ where $\eps_0$ is defined. Then, on the
one hand, we have $\eps_n < \eps_0$, and on the other hand,
$(p_n, q_n) \not\in \Delta(\eps_n)$, therefore $\eps_n > \eps_0$.
This contradiction proves the proposition.
\end{proof}

To prove the existence of periodic canards, we need to define
a mapping $W_\eps \colon \overline{\Pi} \to \reals^2$, and to do that
we need an auxiliary time moment $s_\eps$, which is associated
with the point when the solution either returns to the vicinity
of the point $(p^*, q^*, h_0)$, or deviates from the repulsive
curve $\Gamma_r$ before or after reaching this set.

Recall that $\Omega_0$ denotes a vicinity of zero, where the
coordinate system is curvilinear, and relations
\eqref{e:pqr-cond}--\eqref{e:pqr-a1} are valid, and $\Omega_a$
and $\Omega_r$ are vicinities of $\Gamma_a$ and $\Gamma_r$
where the representations \ref{e:ha-hr} are defined.
By repeating the proofs of Propositions \ref{p:heps-ha} and \ref{p:stable}
in backward time, we can select a constant $C$ such that
$p_\eps - p_r, q_\eps - q_r, h_\eps - h_r \le C \eps$ for any
initial condition $(p_\eps(\sigma), q_\eps(\sigma)) = (p_0, q_0) \in \Pi(2 \alpha)$,
$h_\eps(\sigma) = h_0$, while $w_\eps \not\in \Omega_0$.
Denote
\[
\Omega_\eps^- {}= \{ (p, q, r) \colon h \le h_r(p, q) - 2 C \eps \} \cup \Omega_r \cup (\reals^2 \setminus \Omega_0).
\]
The definition of the time moment $s_\eps(p_0, q_0)$ is divided into several
possible cases:

\begin{enumerate}
\item If $h_\eps(t_1, p_0, q_0) \not\in U_r(2 \gamma(\eps))$ for some
$t_1 < \sigma + 3 \alpha$, then
$s_\eps(p_0, q_0) = \sigma + 2 \alpha$.
\item If $w_\eps(t_1, p_0, q_0) \in \Omega_\eps^-$ for some
$t_0 \le t_1 < \sigma - 3 \alpha$, where $t_0$
the moment when $w_\eps$ exits $\Omega_0$, then
$s_\eps(p_0, q_0) = \sigma - 2 \alpha$.
\item If $w_\eps \not\in \Omega_\eps^-$ for $t_0 \le t \le \sigma + 3 \alpha$,
then $s_\eps(p_0, q_0) = \sigma + 2 \alpha$.
\item Otherwise, the solution begins to fall back to
the attractive part of $S_0$ at the moment $t_1 \in
[\sigma - 3 \alpha, \sigma + 3 \alpha]$.
In this case let $s_\eps(p_0, q_0)$ be the first moment $t_2$ after
$t_1$ such that $h_\eps(t_2, p_0, q_0) = h_0$, or $\sigma + 2 \alpha$,
whichever the smallest, or $\sigma - 2 \alpha$,
whichever the greatest.
\end{enumerate}

\begin{proposition} \label{p:s-cont}
If $\eps$ is sufficiently small, then the moment of time $s_\eps(p_0, q_0)$ is
continuous on $\Pi$ with respect to $p_0$ and $q_0$.
\end{proposition}

\begin{proof}
Let $(p_0, q_0) \in \Pi$ and prove that $s_\eps$ is continuous at the point
$(p_0, q_0)$. Consider a point $(\hat p_0, \hat q_0)$ close to $(p_0, q_0)$
in all four cases from the definition of $s_\eps$:
\begin{enumerate}
\item
$h_\eps(t_1, p_0, q_0) \not\in U_r(2 \gamma(\eps))$ for some
$t_1 < \sigma + 3 \alpha$. The intersection between
$w_\eps$ and $\partial U_r(2 \gamma(\eps))$ is transversal,
thus by virtue of continuous dependence of the solution on
initial values there exists a moment $\hat t_1$ close to $t_1$
such that $h_\eps(\hat t_1, \hat p_0, \hat q_0) \not\in U_r(2 \gamma(\eps))$.
Therefore, $s_\eps(\hat p_0, \hat q_0) = \sigma + 2 \alpha$.

\item
$w_\eps(t_1, p_0, q_0) \in \Omega_\eps^-$ for some
$t_0 \le t_1 < \sigma - 3 \alpha$, where $t_0$
the moment when $w_\eps(t, p_0, q_0)$ exits $\Omega_0$, then
$s_\eps(p_0, q_0) = \sigma - 2 \alpha$. If $t_1 = t_0$,
then $w_\eps(\hat t_0, \hat p_0, \hat q_0) \in \Omega_\eps^-$
at the moment $\hat t_0$ when $w_\eps(t, \hat p_0, \hat q_0)$
exits $\Omega_0$ due to continuous dependence on initial values,
and $s_\eps(\hat p_0, \hat q_0) = \sigma - 2 \alpha$.

If $t_1 > t_0$, then the intersection between $w_\eps$ and $\omega_\eps^-$
is trasversal at the moment $t_1$, thus there exists $\hat t_1$
close to $t_1$ such that $w_\eps(\hat t_1, \hat p_0, \hat q_0) \in \Omega_\eps^-$,
and again $s_\eps(\hat p_0, \hat q_0) = \sigma - 2 \alpha$.

\item
$w_\eps(t, p_0, q_0) \not\in \Omega_\eps^-$ for $t_0 \le t \le \sigma + 3 \alpha$.
Then $w_\eps(t, \hat p_0, \hat q_0) \not\in \Omega_\eps^-$ for
$\hat t_0 \le t_1 \le \sigma + 3 \alpha$, and
$s_\eps(\hat p_0, \hat q_0) = \sigma + 2 \alpha$.

\item
$h_\eps(t_2, p_0, q_0) = h_0$. The intersection between $w_\eps$
and the plane $h = h_0$ is also tranvsersal, as above, thus there
exists a moment $\hat t_2$ close to $t_2$ such that
$h_\eps(\hat t_2, \hat p_0, \hat q_0) = h_0$. Therefore,
$s_\eps$ is continuous at $(p_0, q_0)$.
\end{enumerate}

In all cases, $s_\eps$ is continuous.
\end{proof}

\begin{proposition} \label{p:s-st-unst}
If $(p_0, q_0)$ is stabilizing, then $s_\eps = \sigma - 2 \alpha$ for
sufficiently small $\eps$. If $(p_0, q_0)$ is destabilizing, then
$s_\eps = \sigma + 2 \alpha$.
\end{proposition}

\begin{proof}
Let $(p_0, q_0)$ be stabilizing. According to Proposition \ref{p:stable},
$w_\eps(t_1)$ is $\eps$-close to $w_a(t_1)$ at the moment when $w_\eps$
exits $\Omega_0$. Therefore, $w_\eps(t_1) \in \Omega_\eps^-$, and Case 2
from the definition of $s_\eps$ takes place.

Let now $(p_0, q_0)$ be destabilizing. According to Proposition \ref{p:unstable},
there exists a moment $t_1$ such that $w_\eps(t_1) \not\in U_r(\gamma(\eps))$.
Thus, either the solution exits the set $U_r(2 \gamma(\eps))$ and Case 1
takes place, or it remains in the set $U_r(2 \gamma(\eps)) \setminus U_r (\gamma(\eps))$,
thus never entering $\Omega_\eps^-$, and Case 3 takes place.
\end{proof}

\begin{proposition} \label{p:s-u-v}
Let Case 4 from the definition of $s_\eps$ hold for $(p_0, q_0)$,
that is, there exists a time moment $t_2$ close to $\sigma$ such that
$h_\eps(t_2, p_0, q_0) = h_0$. Then $v(p_0, q_0)$ and
$u(p_\eps(t_2), q_\eps(t_2))$ are $\eps$-small, and $t_2$ is
$\eps$-close to $\sigma + u(p_0, q_0) - v(p_\eps(t_2), q_\eps(t_2))$.
\end{proposition}
\begin{proof}
The fact that $v(p_0, q_0)$ is $\eps$-small is a direct corollary
of Proposition \ref{p:case3}, and $\eps$-smallness of
$u(p_\eps(t_2), q_\eps(t_2))$ is an equivalent statement in
backward time.

Consider now the vicinity $\Omega_0$ of the origin. The solution
$w_\eps$ enters it $\eps$-close to $\Gamma_a$, which has $q < 0$ at
the point of entry, and exits $\eps$-close to $\Gamma_r$, which has $q > 0$.
In this vicinity $Q > 0$, thus there exists a unique time moment
$\hat t$ when $q_\eps = 0$, $w_\eps \in \Omega_0$. It is sufficient
to show that $\hat t$ is $\eps$-close to $u(p_0, q_0)$; $\eps$-closeness
of $t_2 - \hat t$ to $\sigma - v(p_\eps(t_2), q_\eps(t_2))$ can be shown
by repeating the same steps for the system \eqref{e:main-pq} in
backward time.

Suppose the contrary, that is, there exists a sequence $\eps_n \to 0$,
$(p_n, q_n) \in \Pi(\alpha)$, such that $\abs{\hat t_n - u(p_0, q_0)} > t_3 > 0$.
Let $(\hat p, \hat q, \hat t)$ be a limit point of $(p_n, q_n, \hat t_n)$,
and let, for example, $\hat t \le u(p_0, q_0) - t_3$. Then,
according to Proposition \ref{p:stable} applied at the
moment $\hat t$, $q_\eps(\hat t, \hat p, \hat q)$ should be $\eps$-close
to $q^*(\hat t + u(p_0, q_0)) < 0$, however, for sufficiently
small $\eps_n$ and $(p_n, q_n, \hat t_n)$ sufficiently close to
$(\hat p, \hat q, \hat t)$, $q_\eps(\hat t_n, p_n, q_n) = 0$,
thus we arrive at a contradiction. In the case
$\hat t \ge u(p_0, q_0) + t_3$ we apply Proposition \ref{p:stable}
to an appropriate time moment $t_4 < u(p_0, q_0)$ where
$q^*(t_4 + u(p_0, q_0)) > -C$, and $q_\eps(t_4, p_n, q_n) < -C$;
this moment exists because $0 < Q_1 < Q(p, q, h) < Q_2$ in $\Omega$.
\end{proof}

Using the moment $s_\eps(p_0, q_0)$, we define a mapping
$W_\eps$ of the set $\overline{\Pi}(\alpha)$ into the plane $(p, q)$.
The definition is divided into the following two cases:

\smallskip
\noindent \textbf{Case 1.} If
\[
\sigma - 3 \alpha / 2 < s_\eps(p_0, p_0) < \sigma + 3 \alpha / 2,
\]
then
\[
W_\eps(p_0, p_0) = (p_\eps(s_\eps, p_0, q_0), q_\eps(s_\eps, p_0, q_0)).
\]
If we identify the plane $(p, q)$ with the two-dimensional subspace
\[
P_0 = \{ (p, q, h_0) \colon p, q \in \reals \}
\]
of the phase space of system \eqref{e:main-pq}, then in Case~1 the
value $W_\eps(p_0, q_0)$ coincides with the intersection of the
trajectory \eqref{e:sol-eps} with $P_0$, as long as the
corresponding intersection time is close to $\sigma$.

\smallskip
\noindent \textbf{Case 2.} If
\[
3 \alpha / 2 \le \abs{s_\eps(p_0, q_0) - \sigma} \le 2 \alpha,
\]
then
\[
W_\eps(p_0, q_0) = \frac{2 \alpha - \abs{s_\eps - \sigma}}{\alpha / 2}
        (p_\eps(s_\eps, p_0, q_0), q_\eps(s_\eps, p_0, q_0))
    + \frac{\abs{s_\eps - \sigma} - 3 \alpha / 2}{\alpha / 2} w^*(s_\eps).
\]
This means that $W_\eps(p_0, q_0)$ is a continuous convex combination of the
intersection of the trajectory \eqref{e:sol-eps} with $P_0$ and the point
$w^*(s_\eps)$ as long as the discrepancy $\abs{s_\eps - \sigma}$ between the
corresponding intersection moment $s_\eps$ and $\sigma$ in the the range from
$3 \alpha / 2$ to $2 \alpha$. Moreover, if the discrepancy equals $\alpha$, then the
definition is consistent with Case~1; if the intersection moment equals
$\sigma \pm 2 \alpha$, then $W_\eps$ coincides with $w^*(\sigma \pm 2
\alpha)$.

\begin{lemma} \label{l:w-continuous}
If $\eps$ is sufficiently small, then mapping $W_\eps(x_0, y_0)$ is
continuous with respect to $(x_0, y_0)$.
\end{lemma}

\begin{proof}
Follows from the definition of $W_\eps$ and Proposition \ref{p:s-cont}.
\end{proof}

The next lemma establish correspondence between the fixed points of
$W_\eps$ and the periodic solutions of \eqref{e:main-pq}.

\begin{lemma} \label{l:fixed-point}
Let $(\hat p, \hat q) \in \Pi(\alpha)$ be a fixed point of the mapping
$W_\eps$. Then the solution $w_\eps(t, \hat p, \hat q)$ is periodic.
\end{lemma}

\begin{proof}
If suffices to show that if $(\hat p, \hat q)$ is a fixed point, then Case 1
from the definition of $W_\eps$ holds, so that the value of $W_\eps$
corresponds to a point on the trajectory $w_\eps(t, \hat p, \hat q)$ and is
not adjusted, as in Case 2. Note that if Cases 1--3 from the definition of
$s_\eps$ take place, then $W_\eps(\hat p, \hat q) \not\in \Pi$. Thus,
there exists a moment $t_2$ close to $\sigma$ such that
$h_\eps(t_2, \hat p, \hat q) = h_0$. The same argument shows that
$s_\eps = t_2$, otherwise $W_\eps(\hat p, \hat q) \not\in \Pi$.

Suppose that Case 2 holds, and let, for example, $s_\eps(\hat p, \hat q) \ge
\sigma + 3 \alpha / 2$. Denote $(\bar p, \bar q) = w_\eps(s_\eps, \hat p, \hat q)$.
Proposition \ref{p:s-u-v} shows that $v(\hat p, \hat q)$ and $u(\bar p, \bar q)$
are $\eps$-small, and
\[
3 \alpha / 2 \le s_\eps(\hat p, \hat q) - \sigma = u(\hat x, \hat y) -
v(\bar x, \bar y) + \gamma(\eps),
\]
where $\gamma(\eps)$ denotes an $\eps$-small value. Recall that
$\abs{u(\hat x, \hat y)} \le \alpha / 2$. Thus,
\[
v(\bar x, \bar y) < - \alpha + \gamma(\eps) < - \alpha / 2.
\]
Note that the linear combination in the definition of $W_\eps$ moves the point
$(\bar p, \bar q)$ in the direction of the point $(0, - 2 \alpha)$ in
$(u, v)$-coordinates, thus further decreasing $v(\bar p, \bar q)$, therefore
we obtain $v(\hat p, \hat q) < -\alpha / 2$, which
is impossible because $v(\hat p, \hat q)$ must be $\eps$-small.
This contradiction proves that only Case 1 can hold for $(\hat p, \hat q)$.
\end{proof}

Finally, we calculate the rotation of the vector field
$\id - W_\eps$ on $\Pi(\alpha)$.

\begin{lemma} \label{l:rotation}
For sufficiently small $\eps$ the rotation $\gamma(\id - W_\eps, \Pi(\alpha))$
of the vector field $id - W_\eps$ at the boundary of the set $\Pi(\alpha)$
is defined by
\[
\gamma(\id - W_\eps, \Pi(\alpha)) = \sgn(A).
\]
\end{lemma}

\begin{proof}
Let for example,
\[
A > 0.
\]
Consider $R_-$ and $R_+$, the upper and lower sides of
the parallelogram $\Pi(\alpha)$. Propositions \ref{p:unstable} and
\ref{p:stable} imply that
\begin{alignat*}{2}
W_\eps(p, q) &= w^*(\sigma + 2\alpha), &\quad (p, q) &\in R_{-}, \quad \text{and} \\
W_\eps(p, q) &= w^*(\sigma - 2\alpha), &\quad (p, q) &\in R_{+}.
\end{alignat*}
In other words,
\begin{alignat}{2}
u(W_\eps(p, q)) &= 2\alpha,  &\quad (p, q) &\in R_{-} \label{e:f1} \\
u(W_\eps(p, q)) &= -2\alpha, &\quad (p, q) &\in R_{+} \label{e:f2}.
\end{alignat}
Also, Proposition \ref{p:s-u-v} implies that
\begin{equation} \label{e:f3}
\lim_{\eps \to 0} v(W_\eps(p, q)) = 0.
\end{equation}
The relationships \eqref{e:f1}--\eqref{e:f3} imply that in the coordinates
$(u, v)$ the mapping $W_\eps$ on the boundary of $\Pi$ is close to the linear
mapping $L_1(u, v) = (0, -4v)$, thus the mapping $\id - W_\eps$ is close to
(and therefore co-directed with) the linear mapping $L_2(u, v) = (u, 5v)$, and
the result follows from the properties of the rotation number.
\end{proof}

Lemma \ref{l:rotation} implies that the mapping $W_\eps$ has a
fixed point $(\hat p, \hat q)$ on the set $\Pi(\alpha)$, which defines a periodic
solution of \eqref{e:main-pq} according to Lemma \ref{l:fixed-point}.
Moreover, according to Proposition \ref{p:s-u-v}, both $v(\hat p, \hat q)$
and $u(\hat p, \hat q)$ are $\eps$-small, and consequently
$(\hat p, \hat q)$ is $\eps$-close to $(p^*, q^*)$, and also
the minimal period of the solution starting from $(\hat p, \hat q, h_0)$
is $\eps$-close to $\sigma - \tau$. Thus, we have proved the following Statement:

\begin{statement} \label{s:3d}
For any sufficiently small $\eps > 0$ there exists
a periodic solution of system \eqref{e:main-pq}, and thus \eqref{e:main_0},
that passes $\eps$-close to the point $(x^*, y^*)$. The minimal period
$T_{\min}$ of this solution approaches $\sigma - \tau$ as $\eps \to 0$.
\end{statement}

%The next step is to introduce $\lambda > 0$ into the system.
Consider now the system \eqref{e:main} where $\hat{X}, \hat{Y}$
satisfy \eqref{lambda} with a sufficiently small $\lambda > 0$ and
an initial condition $(p_0, q_0, h_0, z_0)$ with $(p_0, q_0) \in
\Pi(\alpha)$, $z_0 \in D$. By repeating the four cases in the
definition of the moment $s_\eps (p_0, q_0)$, we now define the
moment $s_{\eps,\lambda} (p_0, q_0, z_0)$ which has the same
properties as $s_\eps (p_0, q_0).$
%and is close to it if $\lambda$
%is small enough (uniformly for all sufficiently small $\eps$).
Similarly, we define the mapping $W_{\eps,\lambda} (p_0, q_0, z_0)
\colon \Pi\times \bar{D} \to \reals^2$. For sufficiently small
$\eps, \lambda$ this mapping is continuous. Introduce also the
mapping $V_{\eps,\lambda} (p_0, q_0, z_0) \colon \Pi\times \bar{D}
\to \reals^d$ which is defined as follows. Let
$z(t;p_{0},q_{0},h_{0},z_{0})$ denote the $z$-component of the
solution of the system \eqref{e:main} at the corresponding initial
conditions. We define $V_{\eps,\lambda} (p_0, q_0, z_0)$ as closest
to $z(s_{\eps,\lambda} (p_0, q_0, z_0); p_0, q_0, h_0, z_0)$ point
which belongs to the ball of the radius $C \alpha$ centred at
$S_{\sigma - \tau}(z_{0})$ with an appropriate constant $C$.

By construction for small $\alpha, \eps, \lambda$ the vector field
\begin{equation} \label{s}
\id - (W_{\eps, \lambda}, V_{\eps, \lambda})
\end{equation}
is not anti-directed to the vector field
\[
\id - (W_{\eps}, S_{\sigma - \tau})
\]
at the boundary of the domain  $\Pi \times D$. Moreover for small
$\alpha, \eps, \lambda$ each singular point of the field \eqref{s}
generates a required periodic solution. It remains to apply the
product theorem, by which
\[
\gamma(\id - (W_{\eps,\lambda}, V_{\eps,\lambda}), \Pi \times
\bar{D}) = \gamma(\id - S_{\sigma - \tau}, D) \sgn(A) \not= 0.
\]

Thus, Theorem \ref{t:main} is proved.

\subsection{Proof of Theorem \ref{t:chaos-k}}

Here we outline the proof of the Theorem \ref{t:chaos-k} for the case $K = 2$.
It shares many parts of the results from the previous subsection,
so we will only describe the main changes. The proof for arbitrary
values of $K$ is obtained by appropriately adjusting the definitions
of $s^K_\eps$ and $W^K_\eps$ below.

First, we introduce the local coordinates $u_i, v_i$ in the vicinities
of the points $x^*(\tau_i)$, and define the `rectangular' sets
$\Pi_i(\alpha) = \{ (u_i, v_i) \colon \abs{u_i}, \abs{v_i} < \alpha_2 \}$.
The coordinates $u_i, v_i$ can be extended onto a vicinity
$\Omega^K \supset \Pi_i$ of the curve $\Gamma_r$ in such a way
that the transformations $h_i \colon (u_i, v_i) \mapsto (p, q)$
are homeomorphisms. In the coordinates $u_i, v_i$ all the sets
$\Pi_i$ have the same representation $\Pi^* = \{ (u, v) \colon
\abs{u}, \abs{v} < \alpha / 2 \}$.

Suppose that $\sigma_1 < \sigma_2$. Consider a trajectory $w_\eps(t, p_0, q_0)$
of the system \eqref{e:main-pq} with the initial data $(p_0, q_0, h_{0, i}) \in \Pi_i$,
$i = 1, 2$, where $h_{0, i} = (h^*(\tau_i) + h^*(\sigma_i)) / 2$.
Define the time moment $s^K_\eps$ in the following way:

\begin{enumerate}
\item If $h_\eps(t_1, p_0, q_0) \not\in U_r(2 \gamma(\eps))$ for some
$t_1 < \sigma_2 + 3 \alpha$, then
$s^K_\eps(p_0, q_0) = \sigma_2 + 2 \alpha$.
\item If $w_\eps(t_1, p_0, q_0) \in \Omega_\eps^-$ for some
$t_0 \le t_1 < \sigma_1 - 3 \alpha$, where $t_0$
the moment when $w_\eps$ exits $\Omega_0$, then
$s^K_\eps(p_0, q_0) = \sigma_1 - 2 \alpha$.
\item If $w_\eps \not\in \Omega_\eps^-$ for $t_0 \le t \le \sigma_2 + 3 \alpha$,
then $s^K_\eps(p_0, q_0) = \sigma_2 + 2 \alpha$.
\item Otherwise, the solution begins to fall back to
the attractive part of $S_0$ at the moment $t_1 \in
[\sigma_1 - 3 \alpha, \sigma_2 + 3 \alpha]$.
Consider the following two subcases:
\begin{enumerate}
\item If $t_1 \in [\sigma_j - 3 \alpha, \sigma_j + 3 \alpha]$,
let $s^K_\eps(p_0, q_0)$ be the first moment $t_2$ after
$t_1$ such that $h_\eps(t_2, p_0, q_0) = h_{0, j}$, or $\sigma_j + 2 \alpha$,
whichever the smallest, or $\sigma_j - 2 \alpha$,
whichever the greatest.
\item Otherwise, let $s^K_\eps(p_0, q_0) = (\sigma_2 - \sigma_1 - 4 \alpha)
(t_1 - \sigma_1 - 3 \alpha) / (\sigma_2 - \sigma_1 - 6 \alpha) + \sigma_1 + 2 \alpha$.
\end{enumerate}
\end{enumerate}

This moment in time is continuous with respect to $(p_0, q_0)$.
Let us now introduce the Poincar\'e map of the system \eqref{e:main-pq}.

\begin{definition}
If Case 4(a) above holds for $s^K_\eps$, and $h_\eps(s^K_\eps, p_0, q_0) = h_{0, j}$,
then the Poincar\'e map ${\mathcal P}_{pq}$ is defined in the point $(p_0, q_0)$ by
\[
{\mathcal P}_{pq}(p_0, q_0) = (p_\eps(s^K_\eps, p_0, q_0), q_\eps(s^K_\eps, p_0, q_0)).
\]
\end{definition}

Now define the mapping $W^K_\eps \colon \bigcup_i \overline{\Pi}_i \to \reals^2$:

\smallskip
\noindent \textbf{Case 1.} If
\[
\sigma_1 - 3 \alpha / 2 < s^K_\eps(p_0, p_0) < \sigma_2 + 3 \alpha / 2,
\]
then
\[
W^K_\eps(p_0, p_0) = (p_\eps(s^K_\eps, p_0, q_0), q_\eps(s^K_\eps, p_0, q_0)).
\]

\smallskip
\noindent \textbf{Case 2.} If
\[
\sigma_1 - 2 \alpha \le s^K_\eps(p_0, q_0) \le \sigma_1 - 3 \alpha / 2,
\]
then
\[
W^K_\eps(p_0, q_0) = \frac{2 \alpha - \abs{s^K_\eps - \sigma_1}}{\alpha / 2}
        (p_\eps(s^K_\eps, p_0, q_0), q_\eps(s^K_\eps, p_0, q_0))
    + \frac{\abs{s^K_\eps - \sigma_1} - 3 \alpha / 2}{\alpha / 2} w^*(s^K_\eps).
\]

\smallskip
\noindent \textbf{Case 3.} If
\[
\sigma_2 + 3 \alpha / 2 \le s^K_\eps(p_0, q_0) \le \sigma_2 + 2 \alpha,
\]
then
\[
W^K_\eps(p_0, q_0) = \frac{2 \alpha - \abs{s^K_\eps - \sigma_2}}{\alpha / 2}
        (p_\eps(s^K_\eps, p_0, q_0), q_\eps(s^K_\eps, p_0, q_0))
    + \frac{\abs{s^K_\eps - \sigma_2} - 3 \alpha / 2}{\alpha / 2} w^*(s^K_\eps).
\]

\smallskip
The mapping $W^K_\eps$ is continuous with respect to $(p_0, q_0)$, and
the following counterpart of the Lemma \ref{l:fixed-point} holds:

\begin{lemma} \label{l:pmap-defined}
Let $(p_0, q_0) \in {\overline \Pi}_i$ and
$W_\eps(p_0, q_0) = (\hat p, \hat q) \in {\overline \Pi}_j$. Then
the Poincar\'e map ${\mathcal P}_{pq}$ is defined in the point $(p_0, q_0)$,
and ${\mathcal P}_{pq} (p_0, q_0) = (\hat p, \hat q)$.
\end{lemma}

To prove chaoticity of ${\mathcal P}_{pq}$, we use the notion of
$(V, W)$-hyperbolicity from \cite{Homoclinic}.

Fix two positive integers $d_u$, $d_s$ with $d_u + d_s = d$.
Let $V$ and $W$ be bounded, open and convex product-sets
\[
V = V^{(u)} \times V^{(s)} \subset \reals^{d_{u}}\times\reals^{d_{s}},
\quad
W = W^{(u)} \times W^{(s)} \subset \reals^{d_{u}}\times\reals^{d_{s}},
\]
satisfying the inclusions $0 \in V, W$, and let
$g \colon \overline{V} \to \reals^{d_{u}}\times\reals^{d_{s}}$
be a continuous mapping. It is convenient to treat $g$
as the pair $(g^{(u)}, g^{(s)})$ where $g^{(u)} \colon V \to \reals^{d_{u}}$
and $g^{(s)} \colon V \mapsto \reals^{d_{s}}$.

\begin{definition}
The mapping $g$ is \emph{$(V, W)$-hyperbolic}, if the equations
\begin{equation} \label{e:hyp1}
g^{(u)}\left(\partial V^{(u)} \times \overline{V}^{(s)}\right)
\bigcap \overline{W}^{(u)} = \emptyset,
\qquad
g(\overline{V}) \bigcap \left(\overline{W}^{(u)} \times
  (\reals^{d_{s}} \setminus W^{(s)}) \right) = \emptyset
\end{equation}
hold, and
\begin{equation} \label{e:hyp2}
\gamma(g^{(u)}, V^{(u)}) \not= 0.
\end{equation}
Here $\overline{S}$ denotes the closure of a set $S$.
\end{definition}

The first relationship \eqref{e:hyp1} means geometrically
that the image of the `$u$-boundary' $\partial V^{(u)} \times \overline{V}^{(s)}$
of $V$ does not intersect the infinite cylinder
$C = \overline{W}^{(u)} \times \reals^{d_{s}}$;
analogously, the second equality \eqref{e:hyp2} means that the image
of the whole set $g(V)$ can intersect the cylinder $C$
only by its central fragment $\overline{W}^{(u)} \times W^{(s)}$.
Thus the first equation \eqref{e:hyp1} means that the mapping
expands in a rather weak sense along the first coordinate in the
Cartesian product $\reals^{d_{u}} \times \reals^{d_{s}}$,
whereas the second one confers a type of contraction along
the second coordinate (the indices `$(u)$' and `$(s)$' refer
to the adjectives `stable' and `unstable').

The following theorem follows from Corollary 3.1 in \cite{Homoclinic}:

\begin{theorem} \label{t:chaos}
Let $f \colon \reals^{d} \to \reals^{d}$ be a continuous mapping.
Let there exist homeomorphisms $h_{i}$ and product sets
$V_{i}$ such that $h^{-1}_{j} f h_{i}$ is $(V_{i}, V_{j})$-hyperbolic
for all $i, j$, and let the family ${\mathcal U}$ of connected components
of the union set $\bigcup h_{i}(V_{i})$ have more than one element.
Then the mapping $f$ is $\mathcal U$-chaotic.
\end{theorem}

The following lemma establishes the $(\Pi_i, \Pi_j)$-hyperbolicity
of $W^K_\eps$.

\begin{lemma} \label{l:hyperbolic}
For every sufficiently small $\eps$ the mappings ${\hat W^K_\eps}_{i j} = h_j^{-1} W^K_\eps h_i$
are $(\Pi_*, \Pi_*)$-hyperbolic.
\end{lemma}

Lemma \ref{l:hyperbolic} shows that conditions of the Theorem
\ref{t:chaos} are satisfied, thus the mapping $W_\eps$ is $\mathcal
U$-chaotic. The set $S$ in our definition of chaos is invariant,
thus on this set $W_\eps$ coincides with ${\mathcal P}_{pq}$ by
virtue of Lemma \ref{l:pmap-defined}. This establishes the
chaoticity of ${\mathcal P}_{pq}$.

Finally, we consider the full system \eqref{e:main} and define the
time moment $s^K_{\eps, \lambda} (p_0, q_0, z_0)$, the Poincar\'e
map ${\mathcal P} (p_0, q_0, z_0) \colon$ and the mapping
$W^K_{\eps, \lambda} (p_0, q_0, z_0) \colon \bigcup_i
\overline{\Pi}_i \times \bar D_i \to \reals^2$. These definitions
repeat almost literally those of $s^K_\eps$, ${\mathcal P}_{pq}$ and
$W^K_\eps$ with appropriate modifications.

The conditions of the Theorem require that $S_{\sigma_j - \tau_i}
D_i \subset D_j$, which together with Lemma \ref{l:hyperbolic} imply
that mapping
\[
(W^K_{\eps, \lambda}, z(s_{\eps, \lambda} (p_0, q_0, z_0); p_0, q_0,
z_0)),
\]
which is contracting along the $(u_i, z)$ coordinates and expanding
along the $v_i$ coordinate, is $(\Pi_* \times D_i, \Pi_* \times
D_j)$-hyperbolic for all $i, j$. This implies that the Poincar\'e
map ${\mathcal P}$ is chaotic. Thus, Theorem \ref{t:chaos-k} is
proved.

\section{A corollary of the Poincar\'e-Bendixson theorem
and periodic canards} \label{twod}

Consider the differential equation
\begin{eqnarray}\label{e1}
\dot \bx=\bff(\bx, a)
\end{eqnarray}
with a smooth $\bff$.  Here $\bx=(x,y)\in \R^{2}$, and
$a\in[a_{-},a_{+}]$ is a parameter. Let $\Gamma$ be a Jordanian
curve which bounds the open domain $D$. Suppose that for any
$a\in[a_{-},a_{+}]$ there exists a unique equilibrium $\be_{a}\in
D\bigcup\Gamma$, and
\begin{equation}\label{det}
 \det J(\be_{a})>0, \quad a_{-}\le a\le a_{+},
 \end{equation}
 where $J$ denotes the
Jacobian. We also suppose that
\begin{equation}\label{tr}
 \tr J(\be_{a_{-}})\cdot \tr
J(\be_{a_{+}}) <0,
\end{equation}
where $\tr J$ stands for the trace of the Jacobian.
\begin{proposition}\label{pr1}
 Let system \eqref{e1} has no
cycles confined in $D\bigcup \Gamma$ for $a=a_{-},a_{+}.$ Then for
some $a\in (a_{-}, a_{+})$ there exists a cycle of system \eqref{e1}
which is confined in $D\bigcup \Gamma$, and which touches $\Gamma$.
\end{proposition}
Of course, the gist of this statement is in the last three words:
{\em ``\ldots which touches $\Gamma$''.} This proposition is a
corollary of the Poincar\'e-Bendixson theorem, see the next section
for a proof.

We present below four simple examples to demonstrate the role of
Proposition \ref{pr1} in analysis of periodic two-dimensional
canards.

\paragraph*{Example 1.}
%As the first example we
Consider the system
\begin{equation}\label{en}
\dot x=y, \qquad
 \eps\dot y=-x+F(y+a)
\end{equation}
with a small positive $\ve.$  Suppose that $F(0)=0$, $F'(y)<0$ for
$y<0$ and $F'(y)<0$ for $y>0.$ The curve $x=F(y)$ is a slow curve of
system \eqref{en} for $a=0.$ The branch $x=F(y),$ $y<0,$ is the
attractive part of the slow curve, and the branch $x=F(y),$ $y>0,$
is the repulsive part. The origin is the turning point. Periodic
canards are periodic solutions of system \eqref{en} which follow for
a substantial distance the repulsive branch, see Figure \ref{fig0}.
We say that at $a=0$ system \eqref{en} has a periodic canard of
magnitude $\alpha>0$, if to any small $\ve>0$ one can correspond
$a_{\ve}$ and a periodic solution $(x_{\ve, a_{\ve}}(t), y_{\ve,
a_{\ve}}(t))$ of the system $\dot x=y, \ \eps\dot
y=-x+F(y+a_{\ve}),$ such that:
 \begin{equation}\label{ek}
\max \{x_{\ve,a_{\ve}}(t):y_{\ve,a_{\ve}}(t)=0\}= \alpha.
 \end{equation}
 \begin{figure}[htb]
\begin{center}
\includegraphics*[width=9cm]{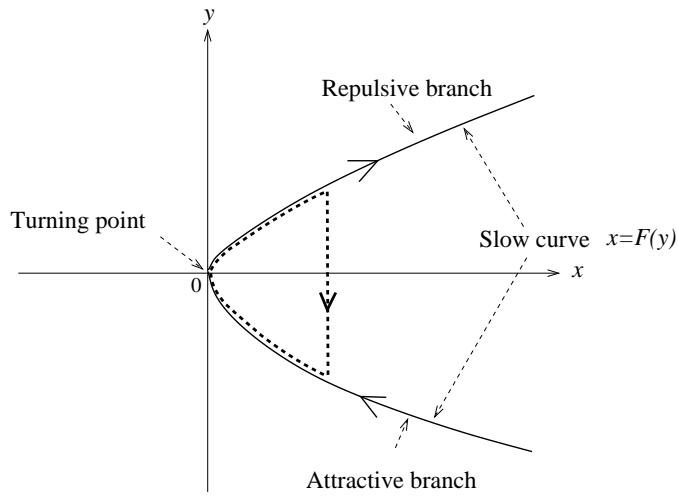}
\end{center}
\caption{Attractive and repulsive branches of the slow curve, and an
example of a periodic canard (dotted line).  } \label{fig0}
\end{figure}
In our case a periodic solution may visit the upper half-plane $y>0$
only traveling along the  repulsive branch of the slow curve. Thus,
this definition is consistent with the informal explanation given
above.
%informally speaking, a periodic canard on the system
%\eqref{en} is a family of solutions which follow from time to time
%the repulsive branch of the slow curve for a prescribed distance
%$\alpha$.
%, and which
%never follow the repulsive branch for longer than this distance
%$\alpha$.

\begin{proposition}\label{pr2}
There exists a periodic canard of system \eqref{en} of any given
magnitude $\alpha>0$.
\end{proposition}

\proof
% Let us show that or small
%$\ve$ the conditions of Proposition \ref{pr1} are satisfied.
%
Note that for any value of $a$ the only equilibrium is given by
\begin{equation}\label{equ}
\be_{a}=(F(a),0).
\end{equation}
Thus
$$
 J(\be_{a})= \left(
 \begin{array}{cc}
 0& 1\\-1/\ve& F'(a)/\ve
 \end{array}
 \right),
$$
and the inequalities
\begin{equation}\label{re}
\det J(\be_{a})=1/\ve>0,\quad \tr J(\be_{a})=F'(a)/\ve <(>)0\
\mbox{\ \rm for\ } a<(>)0
\end{equation} follow.

We choose as the boundaries $a_{-}<0<a_{+}$ any numbers satisfying
\begin{equation}\label{betaE}
-\alpha+F(a_{-})<0 \mbox{\quad \rm and\quad}  -\alpha+F(a_{+})<0.
\end{equation}

\begin{figure}[ht]
\begin{center}
\includegraphics*[width=9cm]{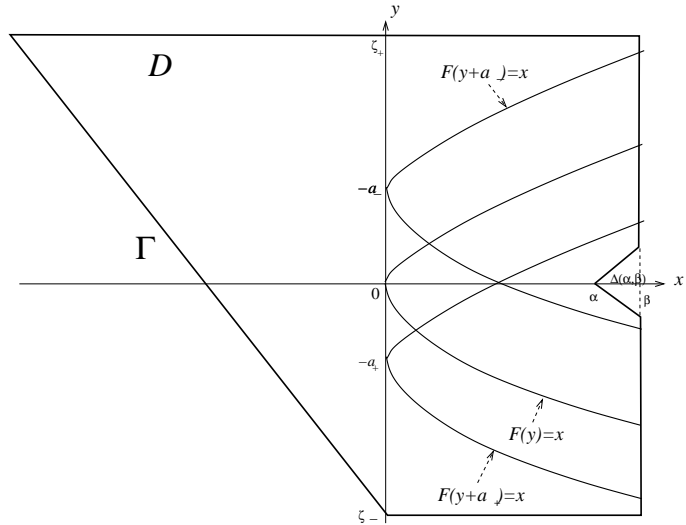}
\end{center}
\caption{The domain $D$ is bounded by the curve $\Gamma$. The
triangle $\Delta(\alpha,\beta)$ belongs to the area where
$-x+F(y+a)<0$ for all $a\in[a_{-},a_{+}].$ The domain $D$ contains
 all slow curves $x=F(y+a),$ $x\in [0, \beta],$
$a\in[a_{-},a_{+}]$. } \label{fig1}
\end{figure}
 Let us describe the domain $D,$ see Figure \ref{fig1}. We choose a number
 $\beta>\alpha$ such that the triangle
$$ %\begin{equation}\label{new2}
\Delta(\alpha,\beta)=\{(x,y): \alpha\le x\le \beta,|y|\le x-\alpha
\} $$ %\end{equation}
belongs to the area where $-x+F(y+a)<0$ for $a_{-}\le a\le a_{+}.$
In particular,
\begin{equation}\label{new1}
-x+F(y+a)<0 \quad \mbox{\rm for}\quad  (x,y)\in
\Delta(\alpha,\beta), \ a\in[a_{-},a_{+}].
\end{equation}
Existence of such $\beta$ follows from \eqref{betaE}. We choose also
 the numbers $\zeta_{-}<0<\zeta_{+}$ satisfying
%$-\beta+F(\zeta_{-}+a_{+}),\beta+F(\zeta_{+}+a_{-})>0.$ Then
\begin{equation}\label{new2}
-x+F(\zeta_{\pm}+a)>0 \quad \mbox{\rm for}\quad  x<\beta, \
a_{-}<a<a_{+}.
\end{equation}
Consider the open quadrangle $Q$ which is bounded from south and
from north by the lines $y=\zeta_{-}$ and $y=\zeta_{+},$ bounded
from east by the line $x=\beta$, and bounded from south-west by the
line $x+y=\zeta_{-}.$ Denote by $D$ the open set $Q\setminus
\Delta(\alpha,\beta)$, and denote by $\Gamma$ the boundary of $D$.

%We denote by $F_{-}(y)$ (correspondingly $F_{+}(y)$ the restrictions
%of the function $F$ to the positive (negative) semi-axis. By
%definition, the inverse functions $F^{-1}_{-}$ and $F^{-1}_{-}$ are
%well defined. Let us denote by
%$$
%a_{-}=F^{-1}_{-}(\alpha/2),\quad a_{+}=F^{-1}_{+}(\alpha/2),\quad
%\eta_{-}=F^{-1}_{-}(\alpha),\quad \eta_{+}=F^{-1}_{+}(\alpha).
%$$
% Let us choose any
%$a_{-},a_{+}$ satisfying $\eta_{-}<a_{-}<0<a_{+}<\eta_{+}$.
%Let us fix $\beta>0$ satisfying the inequalities
%$\beta<\eta_{+}-a_{+},$ and $-\eta_{-}+a_{-}.$ Introduce the
%quadrangle
%$$R=\{(x,y):
%y+\zeta_{-}<x<\beta, \zeta_{-}<y<\zeta_{+}\},$$
% where  $\zeta_{\pm}=F^{-1}_{\pm}(\beta)-a_{\pm}\pm 1, $ and
%the triangle
%$$
%T=\{(x,y): \alpha\le x<\beta,|y|\le x-\alpha \},
%$$

Note that the equilibria $\be_{a},$ $a\in[a_{-},a_{+}],$ belong to
$D\bigcup \Gamma$ by \eqref{equ}; thus $a_{-}<0<a_{+}$ and
\eqref{re} guarantee that \eqref{det} and \eqref{tr} hold.
%\begin{figure}[htb]
%\begin{center}
%\includegraphics*[width=10cm]{fig12.eps}
%\end{center}
%\caption{} \label{fig1}
%\end{figure}
To apply Proposition \ref{pr1} it remains to show that there are no
cycles confined in $D\bigcup \Gamma$ for $a\in\{a_{-},a_{+}\}.$
%Let $\xi$ be
%a negative solution of the equation $F(y)=x_{a_{-}}/2.$

\begin{figure}[ht]
\begin{center}
\includegraphics*[width=10cm]{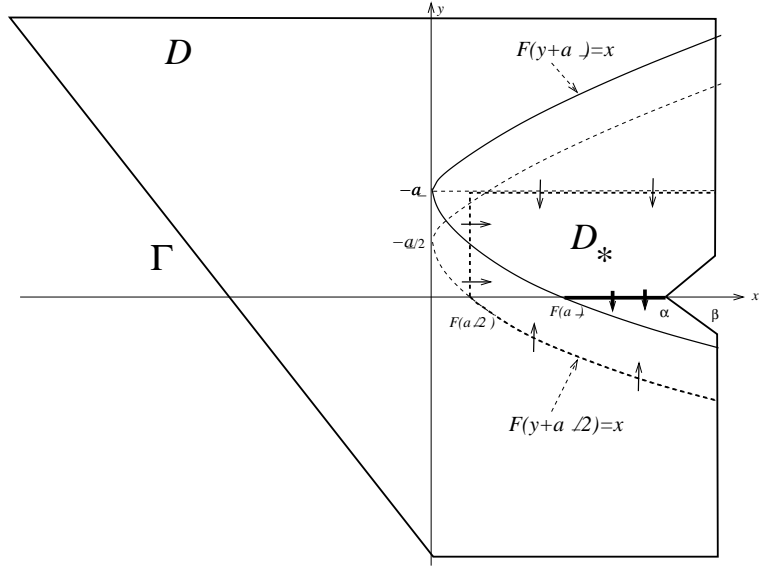}
\end{center}
\caption{The sub-domain $D_{*}\subset D$ is bounded by the bold
dashed line. A trajectory which is confined in $D$ cannot leave the
domain $D_{*}$, once it entered $D_{*}$. Each periodic solution
$\bx_{*}(t)=(x_{*}(t),y_{*}(t))$ which is confined in $D$ must visit
$D_{*}$, because it must cross the bold segment of the axis $y=0$.
The areas are shrinking within $D_{*}$, and thus there are no cycles
there.} \label{figa2}
\end{figure}
Consider the case $a=a_{-}$. Introduce the auxiliary sub-domain
$D_{*}\subset D$ which is bounded from north by the line
$y+a_{-}=0$, from west by the line $x=F(a_{-}/2)$ and from
south-east by the graph of the function $-x+F(y+a_{-}/2)=0,$ see
Figure \ref{figa2}.
%$$
%D_{*}=\{(x,y)\in D:x>\alpha/4,\ x>F(y+F^{-1}_{-}(\alpha/4),\
%y+a_{-}<0\},
%$$
%see Figure \ref{fig1}.
{\em For small $\ve$ a trajectory
$\bx_{\ve}(t)=(x_{\ve}(t),y_{\ve}(t))$ which is confined in $D$
cannot leave the domain $D_{*}$, once it entered $D_{*}$}. To prove
this claim, we note that for the small $\ve$ the velocity vectors
$\dot\bx$ point inward at the part of the boundary of $D_{*}$ which
belongs to $D$. Indeed, at the west boundary
%, for $x=F(a_{-}/2),0<y<a_{-},$
we have $\dot{x}=y>0$; at the north boundary
%, at $y=a_{-},x>x=F(a_{-}/2)$
the inequality $\dot{y}=(-x+F(y+a_{-})/\ve<0$ holds, and at the
south-west boundary the velocity vectors point almost vertically up
for small $\ve$.
 {\em Moreover, each periodic solution
$\bx_{*}(t)=(x_{*}(t),y_{*}(t))$ which is confined in $D$ must visit
$D_{*}.$} Indeed, because $\dot x=y$, the solution $\bx_{*}(t)$ must
visit both half-plane $y<0$ and half-plane $y>0$. Thus $\bx_{*}(t)$
must cross some times the axis $y=0$  from above, i.e., for
$x>x_{a_{-}}=F(a_{-}/2)$; it remains to note that the whole interval
$$\{(x,0): F(a_{-}/2)\le x< \alpha\}$$ belongs to $D_{*}.$

By the italicized parts of the previous paragraph, each cycle which
is confined in $D$ must be confined in $D_{*}.$ However, within
$D_{*}$ the inequality $\tr J(x,y)=F'(y+a_{-})/\ve<0$ holds, the
areas are shrinking, and therefore there are no cycles there. The
case $a=a_{-}$ is completed, and the case $a=a_{+}$ can be
considered analogously in the backward time.

\begin{figure}[ht]
\begin{center}
\includegraphics*[width=10cm]{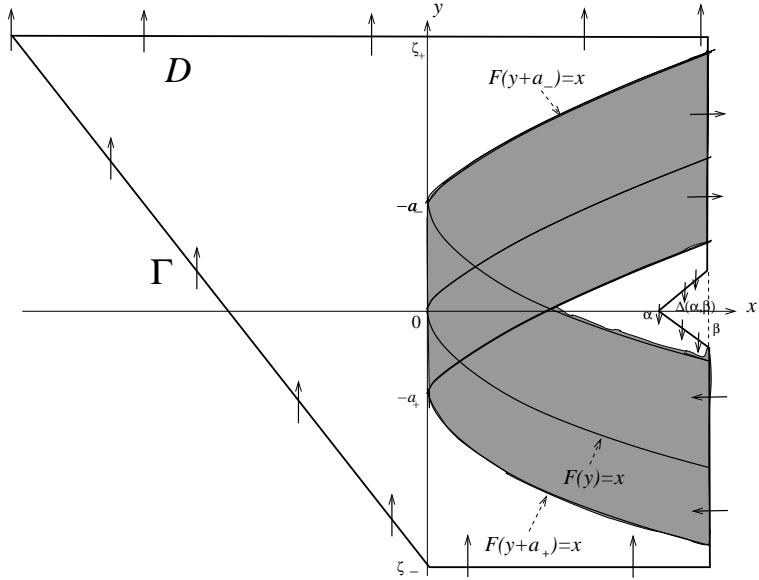}
\end{center}
\caption{A periodic orbit which is confined in $D\bigcup\Gamma$ may
touch $\Gamma$ only at the point $(\alpha,0),$ because at all other
points of $\Gamma$ at least one end of the velocity vector points
strictly outward $\Gamma$.} \label{fig3}
\end{figure}
Thus, by Proposition \ref{pr1}, for any small $\ve>0$ there exists a
periodic solution $(x_{\ve, a_{\ve}}(t), y_{\ve, a_{\ve}}(t))$ whose
trajectory is confined in $D\bigcup \Gamma$ and touches $\Gamma$. On
the other hand, a periodic orbit which is confined in
$D\bigcup\Gamma$ may touch $\Gamma$ only at the point $(\alpha,0),$
see Figure \ref{fig3}: at any other point at least one end of the
velocity vector points strictly outward $\Gamma$. (At the north,
south and south-west parts of the boundary this is due to almost
upward orientation of $\dot\bx$ for small $\ve,$ see \eqref{new1};
at the sides of the triangle $\Delta(\alpha,\beta)$ (apart of the
point $(0,\alpha)$)
--- due to almost downward orientation of $\dot\bx,$ see \eqref{new2}; and at
the vertical fragments of the east boundary --- due to  $\dot{x}=y
\not=0$.) Thus, the family $(x_{\ve, a_{\ve}}(t), y_{\ve,
a_{\ve}}(t))$ is a periodic canard of the required magnitude
$\alpha$, and the proof is completed.

 \begin{figure}[ht]
\begin{center}
\includegraphics*[width=9cm]{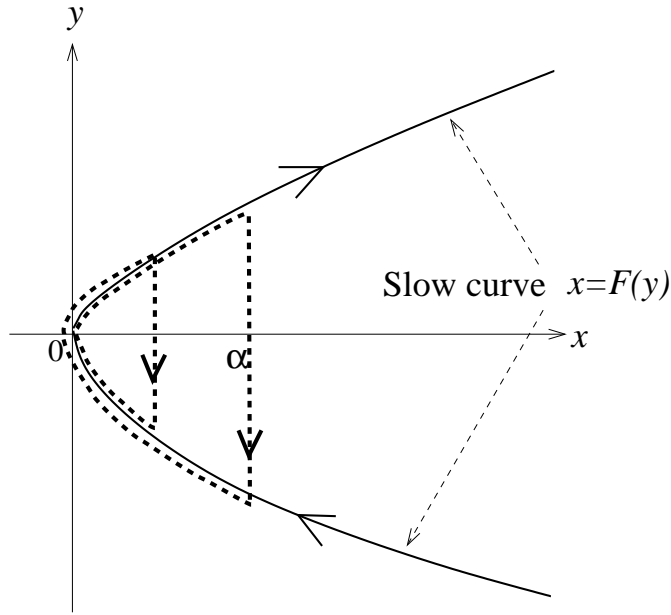}
\end{center}
\caption{Attractive and repulsive branches of the slow curve, and an
example of a periodic canard (dotted line).  } \label{fig0A}
\end{figure}
Statements similar to Proposition \ref{pr2} provide no information
about asymptotic of $a_{\ve}$, and on stability of canards. Still
they could be useful in applications: the canards which existence is
known can be further located and stabilized via a suitable feedback
in a usual way. Note also that we do not guarantee that the a canard
of the magnitude $\alpha$ has only one jump point per period the
structure of a canard may be trickier, see Figure \ref{fig0A}.

\paragraph*{Example 2.}%\newpage
%As the second example we
Consider system \eqref{en} with a bimodal function $f.$ Suppose that
$F(0)=0$, $F'(y)<0$ for $y<0$ and for $y>\mu>0$, whereas $F'(y)>0$
for $0<y<\mu.$  The curve $x=F(y)$ is a slow curve of system
\eqref{en} for $a=0.$ In particular, the branches $x=F(y),$ $y<0,$
and $x=F(y),$ $y>\mu,$ are the attractive parts of the slow curve,
and the branch  $x=F(y),$ $0<y<\mu,$ is the repulsive part. The
origin and the point $(F(\mu),\mu)$ are the turning points. This
modification of the first example is similar to the classical
Lienard equation.

\begin{figure}[ht]
\begin{center}
%\begin{tabular}{cc}
\includegraphics*[width=8cm]{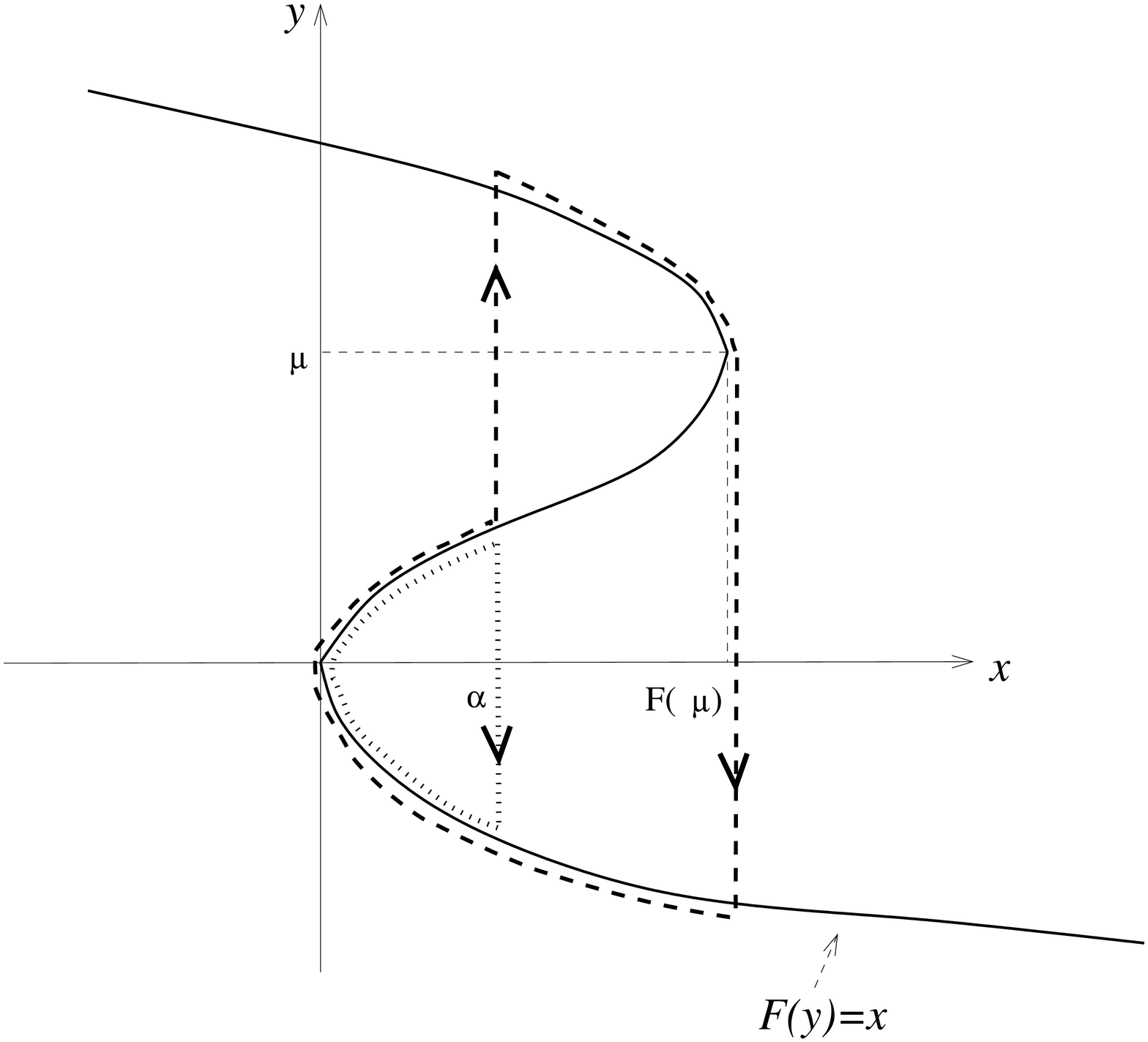}
%\includegraphics*[width=6cm]{bm_D.eps}\\
%a)&b)
%\end{tabular}
\end{center}
\caption {Early canard (dotted) and late canard (dashed for a
bi-modal function $F(x)$.} \label{bm}
\end{figure}
Periodic canards of a magnitude $\alpha$ may exist only for $0<
\alpha\le F(\mu)$. Moreover, there are  two possible structures of a
canard: solution may jump down, or jump up from the unstable part of
the slow curve. We will use wordings early and late canards
correspondingly. Let us give formal definitions.

We say that at $a=0$ system \eqref{en} has an early periodic canard
of magnitude $\alpha>0$, if the relationship \eqref{ek} holds, and
we say that the system has a late periodic canard of magnitude
$\alpha>0$ if instead
 \begin{equation}\label{lk}
\min \{x_{\ve,a_{\ve}}(t): y_{\ve,a_{\ve}}(t)=\mu\}= \alpha.
 \end{equation}

\begin{figure}[ht]
\begin{center}
\begin{tabular}{cc}
\includegraphics*[width=6cm]{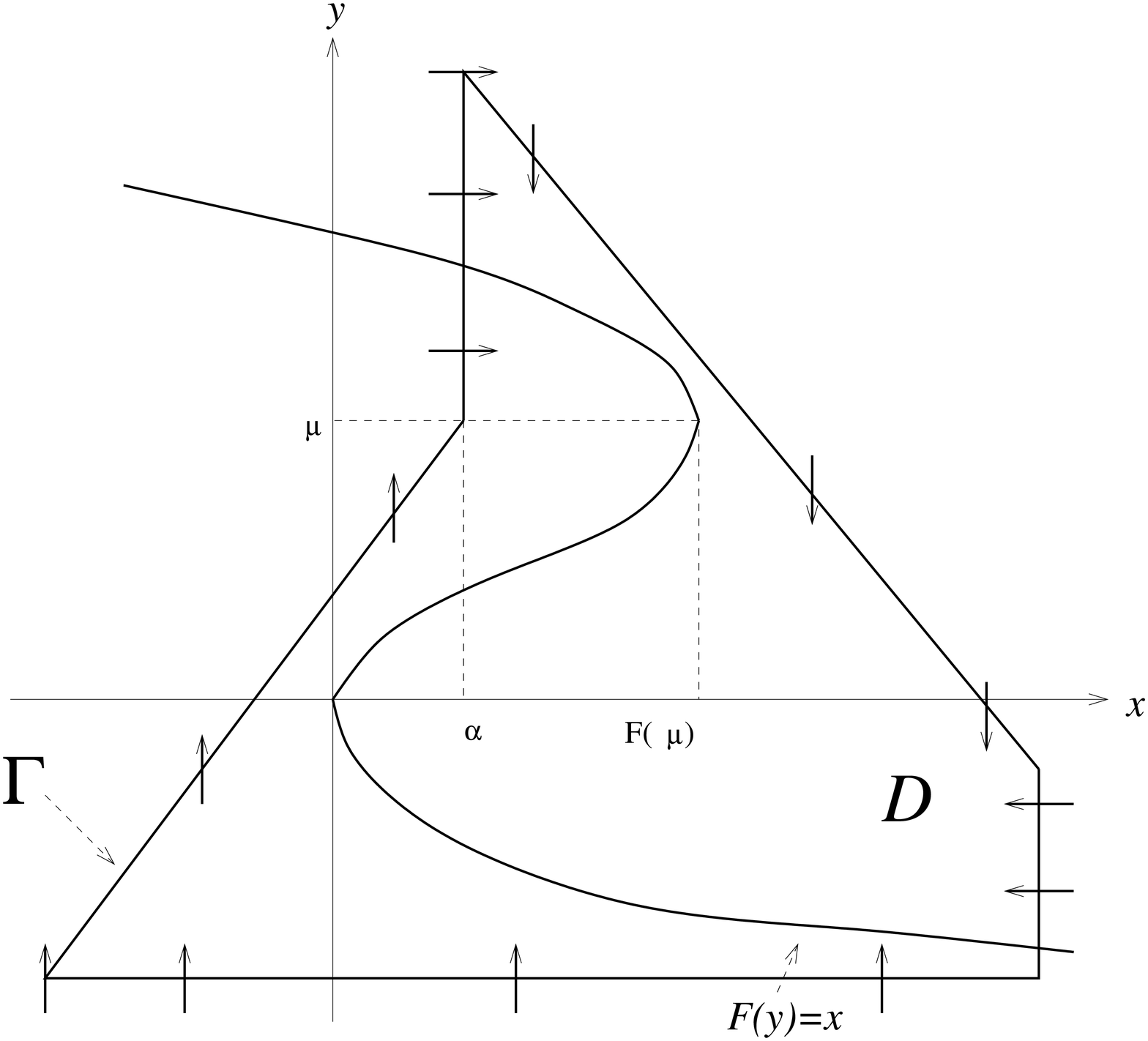} &
\includegraphics*[width=6cm]{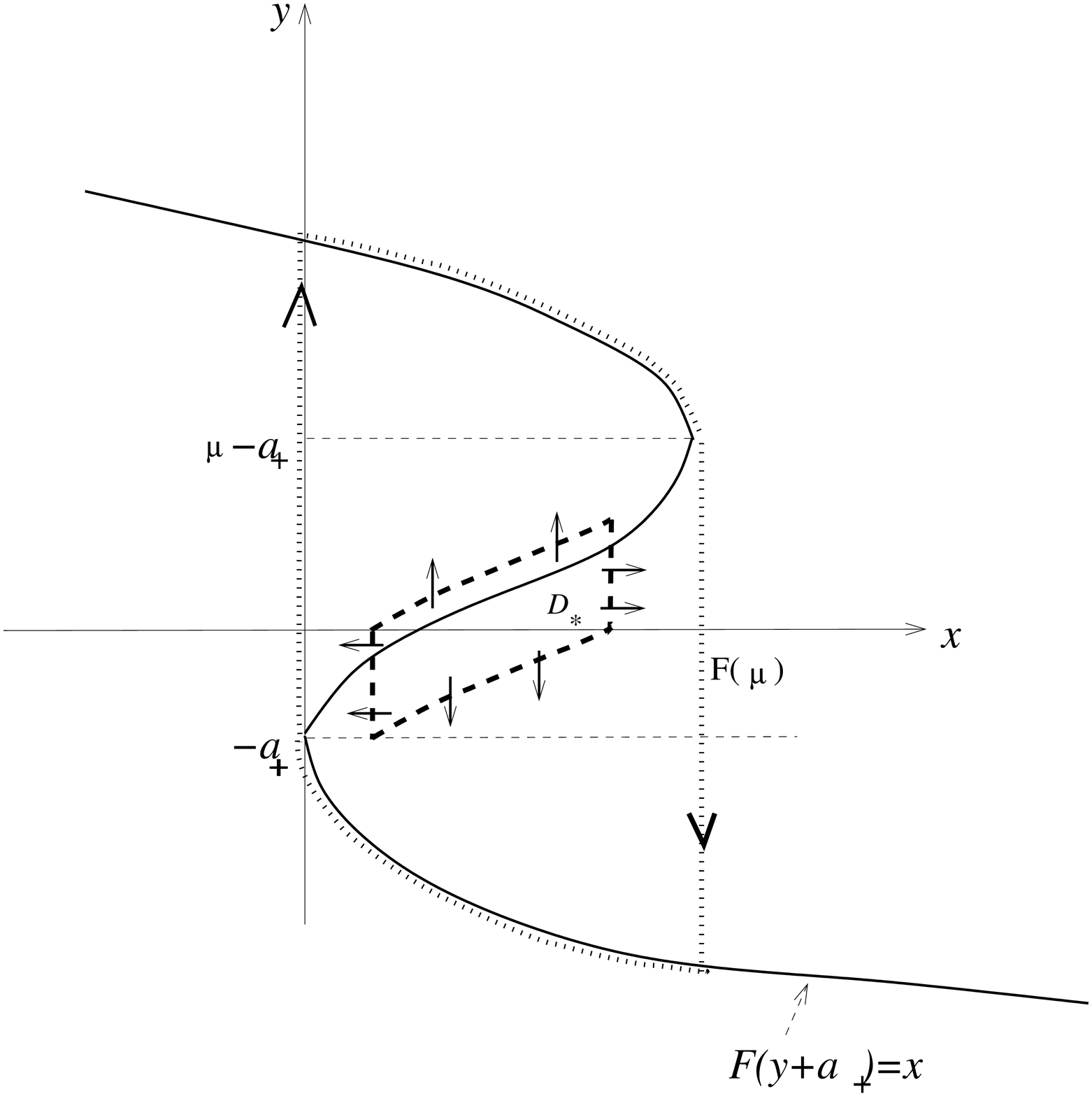}\\
a)&b)
\end{tabular}
\end{center}
\caption{a)  Schematic image of a suitable region $D$ bounded by the
curve $\Gamma$. \newline b) To prove non-existence of cycles at
$a=a_{+}$ we consider region $D_{*}$. This region is a repeller, and
it doesn't embrace any cycle, since the areas are growing within.
Thus any cycle should cross the horizontal axis to the right of
$D_{*}$, and for small $\ve$ must be close to the relaxation cycle
(dotted). However for small $a_{+}$ this relaxation cycle does not
belong to the region $D$.\label{bm2}}
\end{figure}
\begin{proposition}\label{pr3}
There exists an early and a late periodic canards of any given
magnitude $\alpha\in (0,F(\mu)).$
\end{proposition}

Proof is similar to that of Proposition \ref{pr1}. As $a_{\pm}$ one
can chose any small numbers satisfying $a_{-}<0<a_{+}$; the
inequalities \eqref{det},\eqref{tr} are evident. A possible
construction of the region $D$ is given in Figure \ref{bm2}a.
Non-existence of cycles at $a=a_{-}$ may be proven as before. For
non-existence of cycles at $a=a_{+}$ see Figure \ref{bm2}b.

\paragraph*{Example 3.}
\begin{figure}[ht]
\begin{center}
\begin{tabular}{cc}
\includegraphics*[width=6cm]{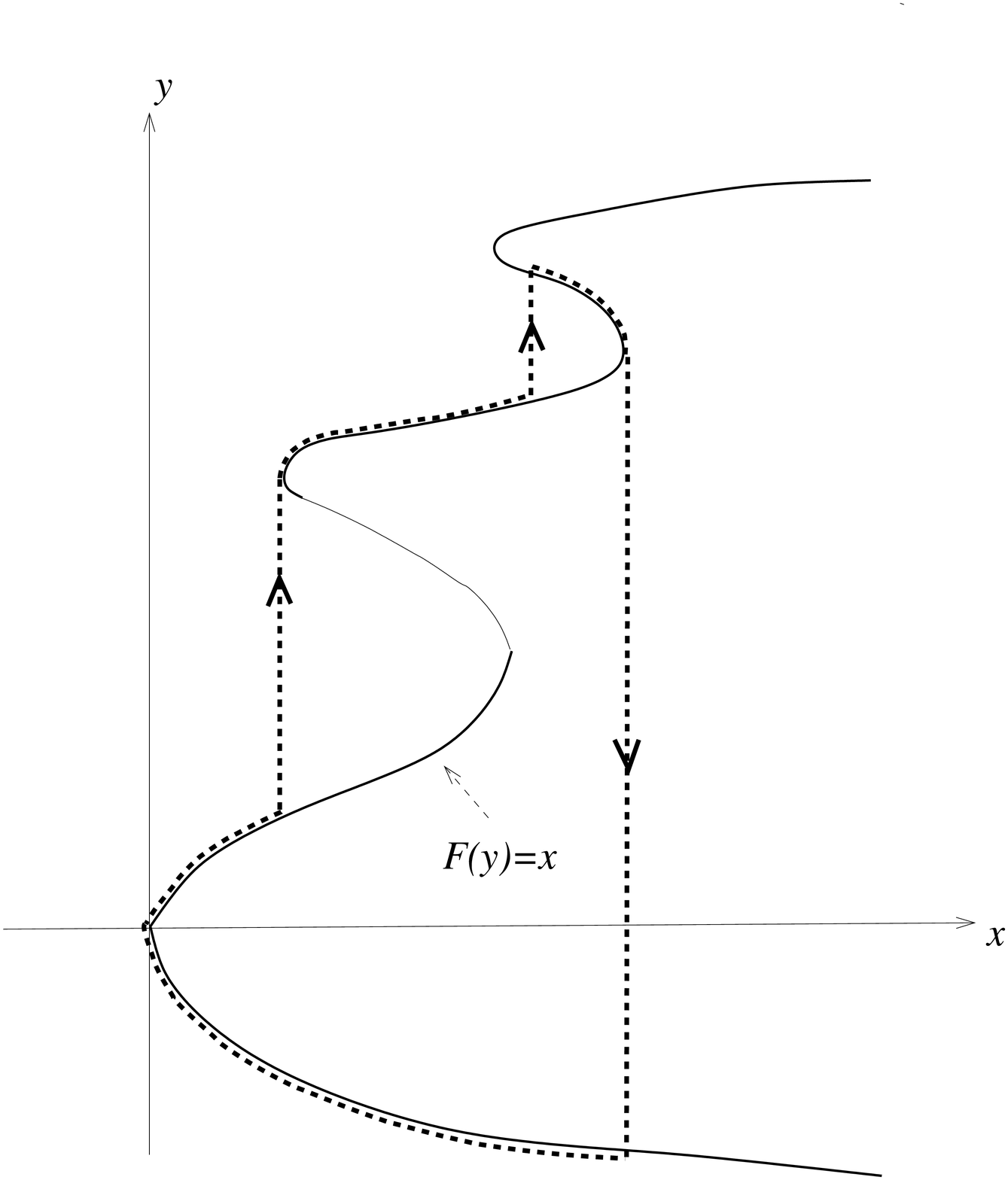}&
\includegraphics*[width=6cm]{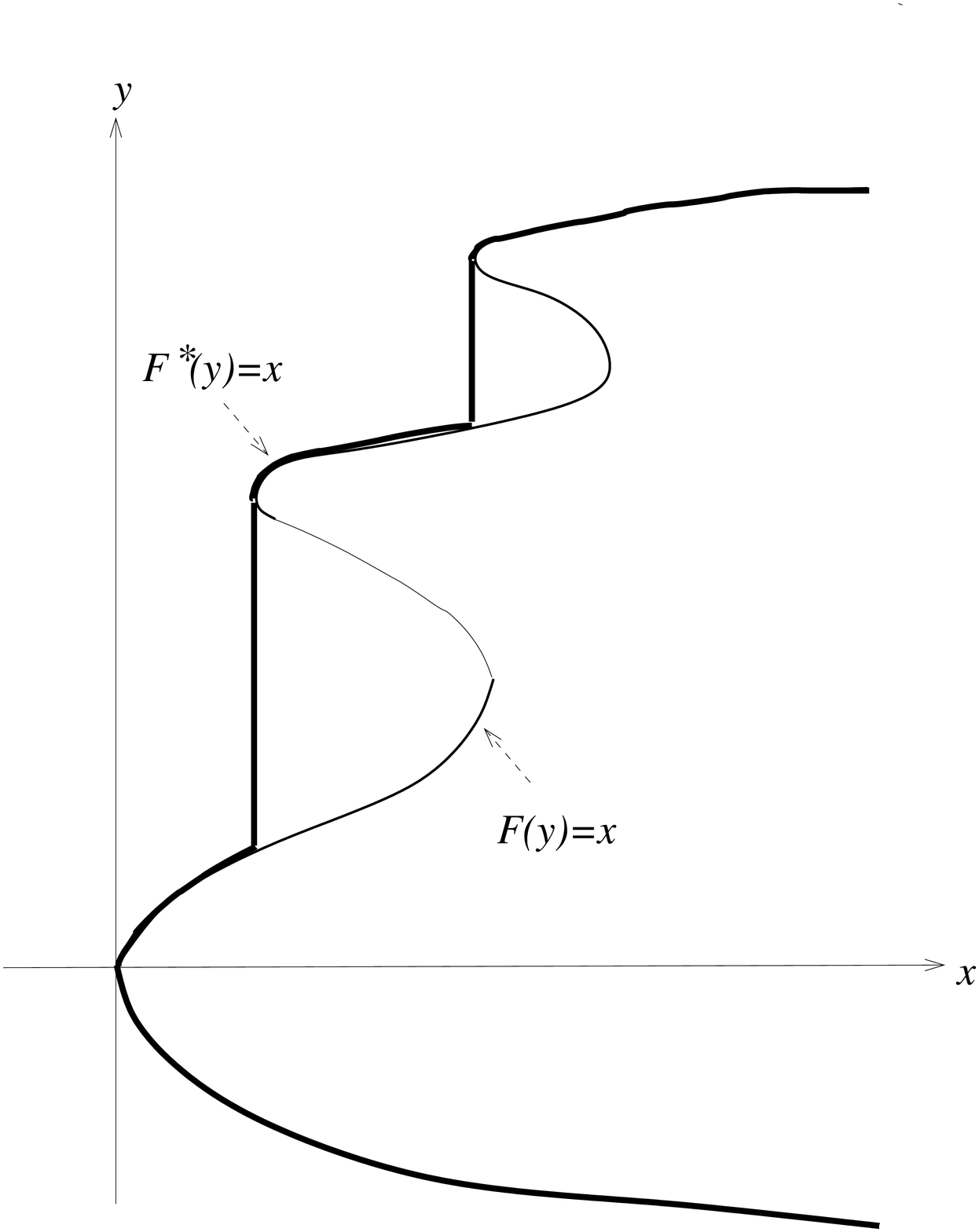}\\
a)& b)
\end{tabular}
\end{center}
\caption{a) An example of the ``super-late'' canard for a multi-mode
function $F(y)$.\newline b) The line $x=F^{*}(y)$ (bold) versus the
line $x=F(y)$. Early $(x_{0},y_{0})$-canards exist for any point
$(x_{0},y_{0})$ located to the right of the line $x=F(y)$; late
 $(x_{0},y_{0})$-canards exist for any point
$(x_{0},y_{0})$ located strictly between the lines  $x=F(y)$ and
$x=F^{*}(y).$} \label{mm}
\end{figure}
As the next example we consider system \eqref{en} where $F(y)$ is a
continuous piece-wise monotone function which satisfies the
following conditions
\begin{equation}\label{mme}
F(0)=0, \quad F(y)>0, \ y\not= 0, \quad
\lim_{y\to\pm\infty}F(y)=\pm\infty.
\end{equation}
We also suppose that al local extrema
 of this function are pairwise
different.

For a given $x_{0}>F(y_{0})$ we say that that at $a=0$ system
\eqref{en} has an early $(x_{0},y_{0})$-periodic canard, if to any
small $\ve>0$ one can correspond $a_{\ve}$ and a periodic solution
$(x_{\ve, a_{\ve}}(t), y_{\ve, a_{\ve}}(t))$ of the system $\dot
x=y, \ \eps\dot y=-x+F(y+a_{\ve}),$ such that
\begin{equation}\label{ekm}
\max \{x_{\ve,a_{\ve}}(t):y_{\ve,a_{\ve}}(t)=y_{0}\}= x_{0}.
\end{equation}
The late canards for the case $x_{0}<F(y_{0})$ are defined
analogously, with the difference that \eqref{ekm} is swapped by
\begin{equation}\label{lk2}
\min \{x_{\ve,a_{\ve}}(t): y_{\ve,a_{\ve}}(t)=y_{0}\}= x_{0}.
\end{equation}

Introduce the auxiliary function
$$
F^{*}= \left\{\begin{array}{rl} \min_{z\ge y} F(z) & y\ge 0,\\[2mm]
\max_{z\le y} F(z) & y< 0.
\end{array}
\right.
$$
\begin{proposition}\label{pr5}
There exists an early $(x_{0},y_{0})$-periodic canard for any
$x_{0}>F(y_{0}$, and there exists a late $(x_{0},y_{0})$-periodic
canard for any $F(y_{0})<x_{0}<F^{*}(y_{0}).$
\end{proposition}

The proof combines the proofs of two previous propositions.

\paragraph*{Example 4.}
As the last example we consider the system
\begin{equation} \dot{x}=F(x,y)=x(p-f(y)), \qquad
\label{gs} \varepsilon\dot{y}=G(x,y,a)=y(-q+x(r+g(y)-a h(y))).
\end{equation}
Here $p,q,r>0$ are given numbers, $f(0)=g(0)=h(0)=0$,
$f'(y),g'(y),h'(y)>0$ for $y\ge 0$, $\varepsilon$ is small, and $a$
is a parameter. Suppose also that
\begin{equation}\label{limE}
\lim_{y\to\infty}g(y)/h(y)=0.
\end{equation}
System \eqref{gs} has been recently used in population dynamics.
Loosely speaking, the functions $r+g(y)$ and $h(y)$ describe
facilitation and competition between predators respectively. The
equation \eqref{limE} means that the competition prevails for denser
populations of predators. An instructive example of the functions
$g(y),h(y)$ is given by
\begin{equation}\label{primer}
g(y)=\alpha_{1}y+\alpha_{2}y^{2}+\ldots +\alpha_{m}y^{m}, \qquad
h(y)=\beta_{1}y^{m+1}+\beta_{2}y^{m+2}+\ldots +\beta_{n}y^{m+n},
\end{equation}
where all coefficients are non-negative, and at least on
$\alpha_{i}$ and at least one $\beta_{j}$ is strictly positive.
Loosely speaking, $\alpha_{i}$ measure intensity of mutual
facilitation between $i+1$ predators, whereas $\beta_{j}$ measure
intensity of competition between $m+i+1$ predators. Another similar
example is given by
\begin{equation}\label{primer2}
g(y)=\int_{0}^{M}v(\alpha)\; d\alpha,
 \qquad
 h(y)=
 \int_{M}^{N}w(\alpha)\; d\alpha
,
\end{equation}
where the weight functions $v(\alpha),w(\alpha)$ are positive and
bounded, and $0<M<N.$
%System \eqref{gs} has been introduced in the context of
%bacteria-phages interaction, and the small parameter $\ve$ manifests
%much higher rate of metabolism for phages.

The system of equation to find ``canard-susceptible'' triplets
$(x_{*},y_{*},a_{*})$ is
$$
F(x,y,a)=0,\quad G(x,y,a)=0, \quad G'_{y}(x,y,a)=0.
$$
In the positive quadrant $x,y>0$ this can be rewritten as
$$
f(y)=a, \quad x(r+g(y)-ah(y))=q, \quad g'(y)=ah'(y).
$$
Since $f(0)=0$ and $f'(y)>0$ for $y\ge 0$, there exists a unique
$y_{*}>0$ which satisfies $f(y)=a$; thus
 $ a_{*}=g'(y_{*})/h'(y_{*}),$ and
 $ x_{*}= q/(r+g(y_{*})-a_{*}h(y_{*})).$
We suppose that $x_{*}$ is positive, that is that the inequality
$$
r+g(y_{*})>a_{*}h(y_{*})
$$
holds.

In the positive quadrant the slow curve is given by
\begin{equation}\label{slow}
x=X(y)=q/(r+g(y)-a_{*}h(y)), \quad 0<y<\eta,
\end{equation}
To avoid non-principal complications we suppose that the function
$g'(y)/h'(y),$ strictly decreases for $y>0;$
%, that is $g'(y)h''(y)>g''(y)h'(y);$
this is always true in the case  \eqref{primer} or \eqref{primer2}.
(For instance, in the case \eqref{primer} we rewrite $g'(y)/h'(y)$
as $g_{1}(y)/h_{1}(y)$ with $g_{1}(y)=g'(y)/y^{m},
h_{1}(y)=h'(y)/y^{m}$; then $g_{1}(y)$ strictly decreases,
$h_{1}(y)$ strictly increases, and the the fraction
$g'(y)/h'(y)=g_{1}(y)/h_{1}(y)$ strictly decreases as required.)
  Then, in particular, the function  $r+g(y)-a_{*}h(y)$ is unimodal,
  and there exist the single positive root $\eta$ of the equation
$r+g(y)-a_{*}h(y)=0.$ Therefore the function \eqref{slow} is
  unimodal for $0<y<\eta$.
 The branch $x=X(y),$  $0<y<y_{*},$
 is repulsive, the branch $x=X,$ $y_{*}<y<\eta,$
is attractive, and $(x_{*},y_{*})$ is the unique turning point.

\begin{figure}[ht]
\begin{center}
\includegraphics*[width=10cm]{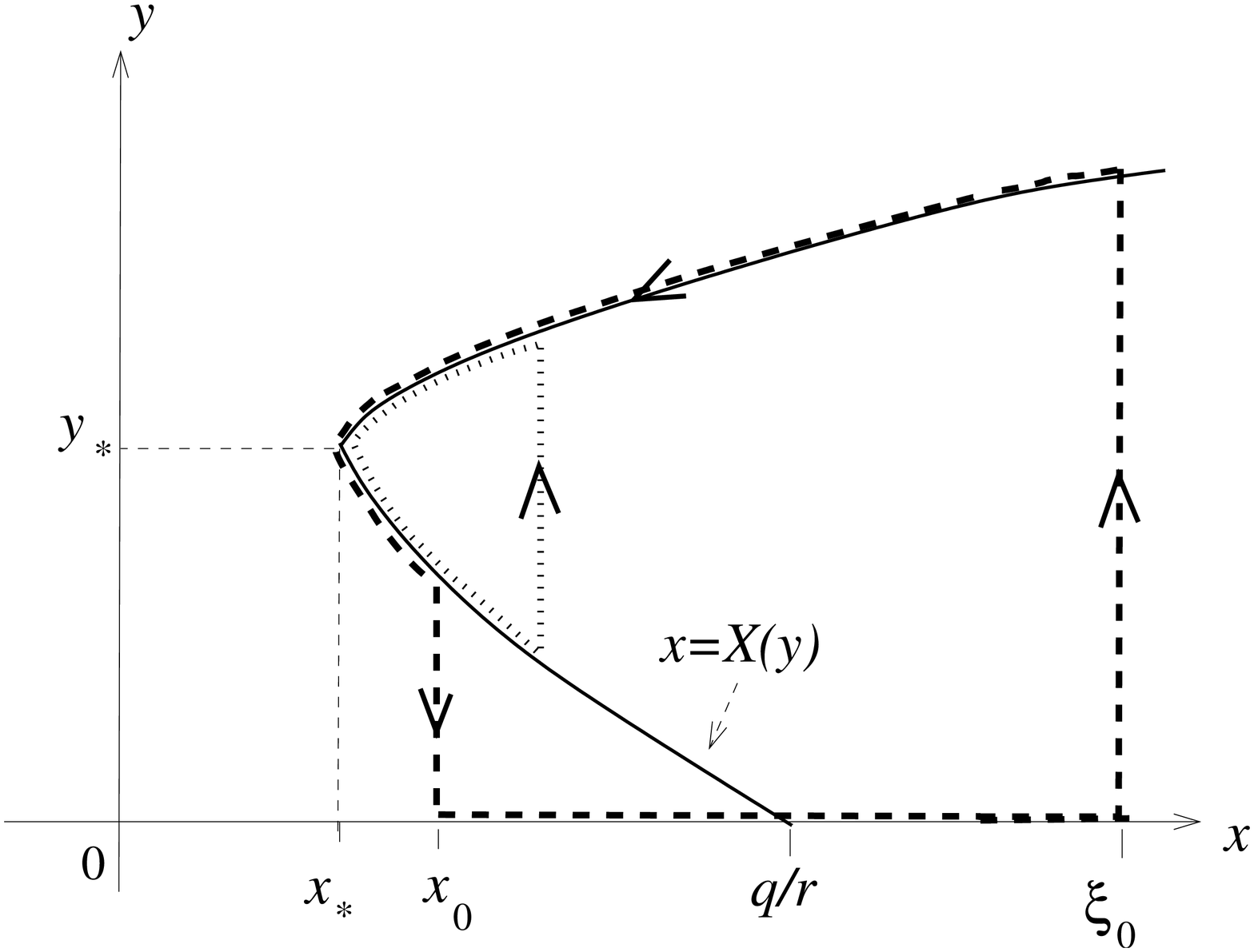} %&
\end{center}
\caption{Early and late canards for the modified Lotka-Vilterra
system. Early canards may exist for $x_{0},y_{0}>0$ satisfying
 $X(y_{0})<x_{0}<X(0),$ and they have the standard structure.
 Late canards, which may exist for $x_{0},y_{0}>0$ satisfying
$0<y_{0}<y_{*},$ $x_{*}<x_{0}<X(y_{0}),$ are more interesting. A
late $(x_{0},y_{0})$-canard exhibits additional delayed loss of
stability phenomenon: after following down at $x\approx x_{0}$ it
follows closely the axis $y=0$ until the point $\xi_{0}$ and then
jumps up to the attractive branch of the slow manifold. %The point
%$\xi_{0}>q/r$  is the solution of the equation
%$\xi^{q}\exp(-r\xi)=x_{*}^{q}\exp(-rx_{*}).$
} \label{l-v}
\end{figure}
For $a\sim a_{*}$ the system may  have two types of canards, see
Figure \ref{l-v}. Early canards may exist for $x_{0},y_{0}>0$
satisfying
 $X(y_{0})<x_{0}<X(0),$ and they have the standard structure.
 Late canards, which may exist for $x_{0},y_{0}>0$ satisfying
$0<y_{0}<y_{*},$ $x_{*}<x_{0}<X(y_{0}),$ are more interesting. A
late $(x_{0},y_{0})$-canard exhibits additional delayed loss of
stability phenomenon: after following down at $x\approx x_{0}$ it
follows closely the axis $y=0$ until a point $x\approx \xi_{0}>q/r,$
and then jumps up to the attractive branch of the slow manifold. The
point $\xi_{0}$  is the solution of the equation
$\xi^{q}\exp(-r\xi)=x_{0}^{q}\exp(-rx_{0}).$ Indeed, close to the
axis $y=0$ the dynamics is governed by the equation
$dy/dx=y(q-rx)/(px)$ whose solutions satisfy the relationship $\ln
y^{p}- \ln x^{q} +rx=const.$

%The corresponding early $(x_{0},y_{0})$-canards, which
%may exist for $y_{0}>0$, $P_{*}(y_{0})<x_{0}<P_{*}(0)$,
% is standard.  The late $(x_{0},y_{0})$-canards which may exist for any $(x_{0},y_{0})$
%satisfying $0<y_{0}<\mu_{*},$ $x_{a}<x_{0}<P(y_{0},a_{*}),$ are more
%interesting, see Figure \ref{l-v}.

\begin{proposition}\label{pr6}
There exists an early $(x_{0},y_{0})$-periodic canard for any
$x_{0},y_{0}>0$ satisfying
 $X(y_{0})<x_{0}<X(0),$  and there exists a late
$(x_{0},y_{0})$-periodic canard for any $x_{0},y_{0}>0$ satisfying
$0<y_{0}<y_{*},$ $x_{*}<x_{0}<X(y_{0}).$
\end{proposition}

\begin{figure}[ht]
\begin{center}
\includegraphics*[width=10cm]{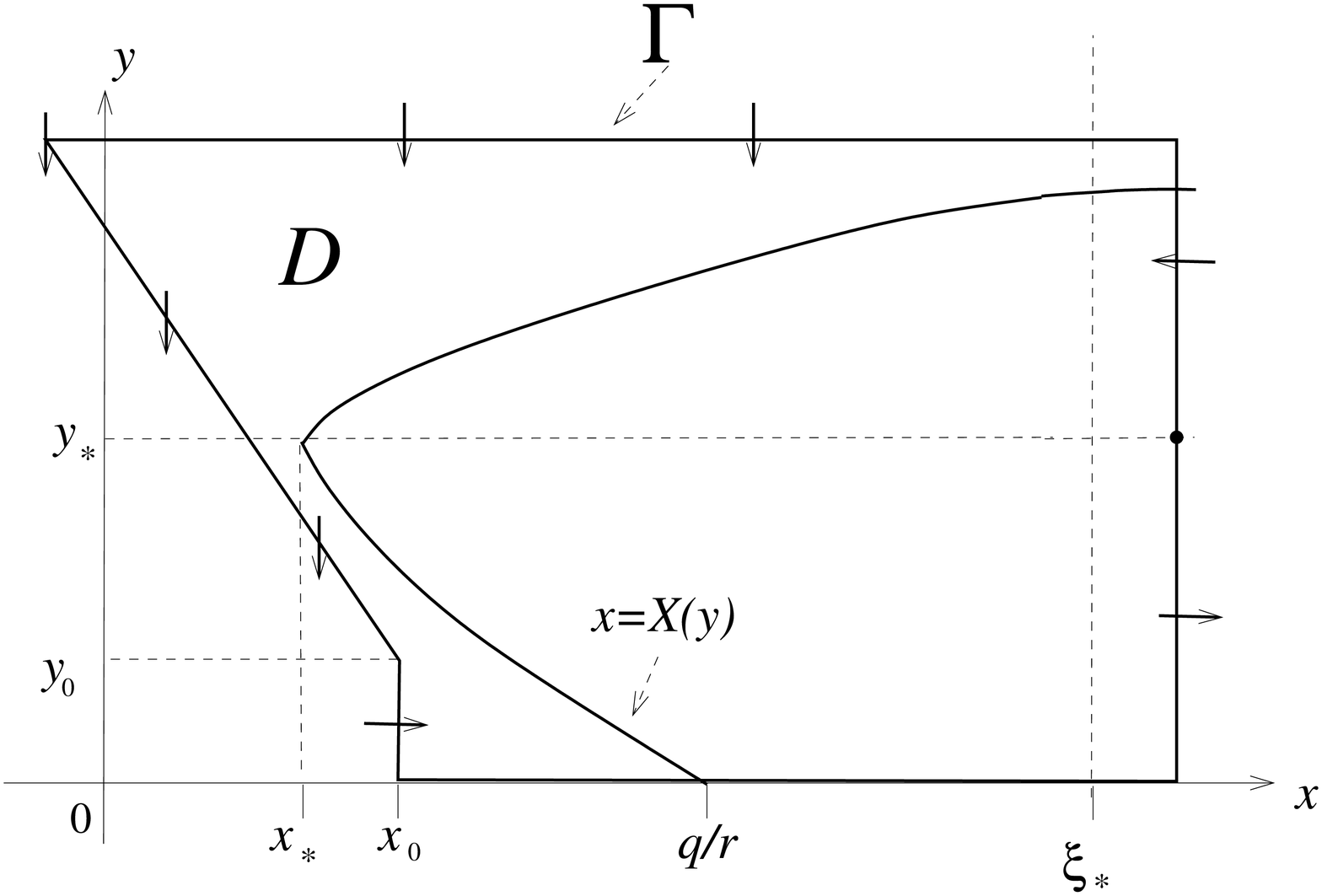} %&
\end{center}
\caption{The region $D$ bounded by the curve $\Gamma$. At last one
endpoint of the velocity vector points strictly outward of $D$ at
all points of $\Gamma,$ except of two points: one is our ``target
point'' $(x_{0},y_{0})$, and another is the bold point at the
eastern bound of $D$. However, there are no cycles which touch
$\Gamma$ at the second point, since the longest possible travel
along the axis $y=0$ ends at the point $2q/r-x_{*},$ which is
located strictly to left of the eastern bound of $D$. Thus, by
Proposition \ref{pr1}, there exists a cycle which is confined in
$D\bigcup \Gamma$, and touches $\Gamma$ at $(x_{0},y_{0}).$}
\label{l-v_D}
\end{figure}

\proof As $a_{\pm}$ we choose any numbers which are sufficiently
close to $a_{*}$ and satisfy $a_{-}<a_{*}<a_{+}$. Note that
$$
 J(\be_{a})= \left(
 \begin{array}{cc}
 0& -x_{a}f'(y_{*})
  \\y_{*}(g(y_{*})-ah(y_{*}))/\ve&
   x_{a}y_{*}(g'(y_{*})-ah'(y_{*}))/\ve
 \end{array}
 \right)
$$
and the inequalities \eqref{det} and \eqref{tr} follow. The
construction of the region $D$ in the case of an early canard is the
same as in the first example, and in the case of the late canard is
explained in Figure \ref{l-v_D}. Here $\xi_{*}$ denotes the unique
positive solution of the equation
$\xi^{q}\exp(-r\xi)=x_{*}^{q}\exp(-rx_{*}).$
 Nonexistence of confined in
$D\bigcup \Gamma$ cycles for $a=a_{-},a_{+}$ can be proven as in the
previous examples.

% Many other statements about existence of two-dimensional canards
%may be modified in this fashion. Statements analogous to Proposition
%\ref{pr2} provide no information about asymptotic of $a_{\ve}$, and
%on stability of canards. Still they could be useful in applications:
%the canards which existence is known can be further located and
%stabilized via a suitable feedback in a usual way.

\section*{Proof of Proposition \protect\ref{pr1}} The
Poincare-Bendixson theorem can be stated in several ways. The
statement that is relevant to the equation \eqref{e1} is the
following. {\em Suppose $S$ is a closed, bounded subset of the
plane; $S$ does not contain any fixed points; and there exists a
trajectory confined in $S$. Then either this trajectory is a closed
orbit, or it spirals toward a closed orbit.}

For a particular value of $a$ a solution $\bx(t)$ of \eqref{e1} is
called directed, if $\bx(0)\in\Gamma$ and $\bx(t)\in \bar D$ for
$t>0$. {\em There exists a directed solution $\bx(t)$ for
$a=a_{+},a_{-}$.} To prove this claim we suppose that $\tr
J_{a_{-}}<0,$ and consider a solution which begins in a sufficiently
small vicinity of $\be_{a_{-}}.$ Then $|\by(t)-\be_{a_{-}}|\to 0$ as
$t\to \infty$, and  $|\by(t)|,$ is bounded from below at $t\le 0$
(because $\be_{a_{-}}$ is a sink due to $\det J_{\be_{a_{-}}},\tr
J_{\be_{a_{-}}}<0$). By the Poincare-Bendixson theorem $\by(t)$ must
leave $D$ in negative time (because there is no cycles at
$a=a_{-}$); in particular, $\by(t)$ touches $\Gamma$ for the first
time at some $t=\tau<0$. It remains to set $\bx(t)=\by(t+\tau)$.
Analogously, using the backward time, we prove that  {\em there are
no $\Gamma$-directed solutions at $a=a_{+}$}.

Denote by $a_{0}\in[a_{-},a_{+})$ the upper bound of
$a\in[a_{-},a_{+})$ for which there exist some directed solutions.
For $a=a_{0}$ a directed solution $\bx_{0}(t)$ also exists by limit
transition. If $\bx_{0}(\cdot)$  is periodic, then the proposition
holds. To finalize the proof {\em we suppose that $\bx(\cdot)$ is
not periodic, and arrive at contradiction.}

By the Poincare-Bendixson  theorem there are only two possibilities:
%\begin{description}\itemsep=0mm
either {(a)} $\bx_{0}(t)$ spirales toward $C$ a cycle $C\subset D,$
or {(b)} $\bx_{0}(t_{n})\to \be_{a_{0}} $  for some
$t_{n}\to\infty$.
%\end{description}

\begin{figure}[htb]
\begin{center}
\includegraphics*[width=10cm]{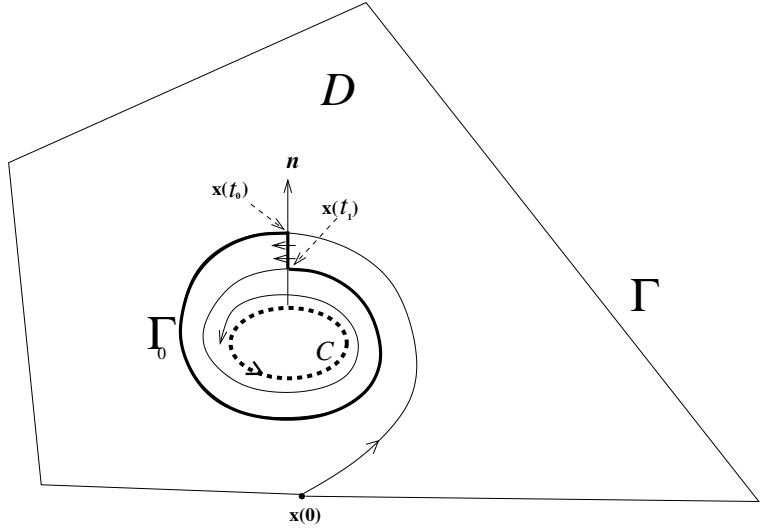}
\end{center}
\caption{The trapping curve $\Gamma_{0}.$} \label{fig4}
\end{figure}
Let $\Gamma_{0}\subset D$ be a Jordanian curve which bounds the open
domain $D_{0}$, and $\tau$ be a positive number. We say that the
pair $\{\Gamma_{0},\tau\}$ is {\em trapping} if simultaneously: the
set $D_{0}\bigcup\Gamma_{0}$ is forward invariant for the equation
$\dot\bx-\bff(\bx,a_{0}),$ and  $\bx(\tau)\in D_{0}$ holds for any
solution satisfying $\bx(0)\in\Gamma_{0}$. If a trapping pair
exists, then for $a$ slightly greater than $a_{0}$ the solutions of
equation \eqref{e1} that begins at $\bx_{0}(0)$ are also attracted
to arbitrary small vicinity of $D_{0}$. That is, there exist a
directed solutions at some $a>a_{0}$. Thus, to arrive at
contradiction it is enough to construct a trapping pair. To this end
in the case (a) we choose a point $\by\in C$ and consider the
corresponding outward normal $\bn$ to $C$. Let $\bar\lambda$
satisfies the relationships $[\by,\by+\bar\lambda\bn]\subset D,$ and
$\bff(\by+\lambda\bn,a_{0})\cdot\bff(\by,a_{0})>0,$ $0\le
\lambda\le\bar\lambda.$ By definition, the solution $\bx_{0}(t)$
crosses the segment $(\by,\by+\bar\lambda\bn]$ infinitely many
times, see Figure \ref{fig4}. Let $t_{0}$ and $t_{1}$ be two
successive moments of such crossings with the corresponding values
$\lambda_{0},\lambda_{2}$. Consider the curve $\Gamma_{0}$ which
consists of the trajectory $\bx_{0}(t),$ $t_{0}<t<t_{1}$, together
with the segment $[\bx_{0}(t_{0}),\bx_{0}(t_{1})]$. Since
$\bx_{0}(t)$ spirales toward $C$, the inequality
$\lambda_{0}>\lambda_{1}$ holds. Therefore,
$\{\Gamma_{0},t_{1}-t_{0}+1\}$ is a trapping pair, and we arrived at
contradiction in the case (a).

By $\tr J(\be_{a_{0}})<0$ the case (b) can be partitioned in turn
into three cases:
%\begin{description} \itemsep=0mm
{(b1)} $\be_{a_{0}}$ is a source; {(b2)} $\be_{a_{0}}$ is a sink;
{(b3)} $\be_{a_{0}}$ is a center in the linear approximation.
%\end{description}
In the case (b1) we immediately arrive at contradiction with the
condition (b). In the case (b2), for $a$ slightly greater than
$a_{0}$, the solution which begins at $\be_{a_{0}}$ is attracted to
a small vicinity of $\be_{a_{0}}$. Thus, there exist directed
solutions for some $a>a_{0}$, which contradicts the definition of
$a_{0}$. It remains to consider the case (b3), which is similar to
the case (a) above. Indeed, consider a segment
$\sigma=(\be_{a_{0}},\be_{a_{0}}+\bz]$ where $\bz$ is close enough
to $\be_{a_{0}}$ to guarantee that $\sigma\subset D$, and that
$\bff(\by,a_{0}),$ $\by\in \sigma,$ is not collinear to $\bz$. (This
can be done because $\be_{a_{0}}$ is a center in the linear
approximation.) By the condition (b)  the solution $\bx_{0}(t)$
crosses the segment $\sigma$ infinitely many times. Let $t_{0}$ and
$t_{1}$ be to successive moments of such crossings. Consider the
curve $\Gamma_{0}$ which consists of the trajectory $\bx_{0}(t),$
$t_{0}<t<t_{1}$, together with the segment
$[\bx_{0}(t_{0}),\bx_{0}(t_{1})]$. By construction the pair
$\{\Gamma_{0},t_{1}-t_{0}+1\}$ is a trapping pair, and we arrived at
contradiction in the case (b3). The proposition is proven.

\section{Non-smooth perturbations\label{last}}

Consider a perturbed system \eqref{en}:
 $$ %\begin{equation}\label{e4}
\dot x=y, \qquad
 \eps\dot y=-x+F(y-a)+\tilde F(x,y,a),
 $$ %\end{equation}
where $\tilde F$ is continuous and small in the uniform norm:
$\sup|\tilde F(x,y,a)|<\delta\ll 1$,  but there is no bounds for on
its derivatives. In this case applicability of usual tools is
doubtful.

\begin{proposition}\label{pr2p}
There exist $\bar\ve,\bar\delta>0$ such that for $0<\ve<\bar\ve,
0<\delta<\bar\delta$ there exists a small $a_{\ve}$ and a periodic
canard$(x_{\ve}(t), y_{\ve}(t))$ of the system $\dot x=y, \ \eps\dot
y=-x+F(y-a_{\ve})+\tilde F(x,y,a_{\ve})$ which satisfies
$\max_{t}\{x_{\ve}(t)\}=b$. The trajectory of this canard approaches
$\Gamma(b)$ as $\ve,\delta\to 0$.
\end{proposition}

\noindent Proof follows from the following modification of
Proposition \ref{pr1}. Consider the equation
\begin{eqnarray}\label{e5}
\dot \bx=\bff(\bx, a)+\tilde\bff(\bx, a).
\end{eqnarray}
Here $\tilde\bff(\bx, a)$ is continuous and uniformly small:
 $\sup\tilde\bff(\bx,a)<\delta\ll 1$, but there is no
 restriction on its derivative. {\em Under conditions of Proposition
 \ref{pr1}
for some $a\in (a_{-}, a_{+})$ there exists a cycle of system
\eqref{e5} which belongs to $D\bigcup \Gamma$, and which touches
$\Gamma$.} Proof is essentially the same as of Proposition
\ref{pr1}.

\section{Acknowledgements}
The authirs are grateful to Professors F. DuMortier, V. Goldshtein,
R. O'Malley, N. Popovic, V. Sobolev for the useful and pleasant
discussions.


\begin{thebibliography}{99}

\bibitem{dei}
K.~Deimling,
\emph{Nonlinear functional analysis}, Springer, 1980.

\bibitem{geom}
M.~A.~Krasnosel'skii and P.~P.~Zabreiko,
\emph{Geometrical Methods of Nonlinear Analysis}, Springer, 1984.

\bibitem{mawhin}
%A.~M.~Krasnosel'skii, A.~V.~Pokrovskii,
%``Large subharmonics of pendulum-like equations'', in
\emph{The first 60 years of Nonlinear Analysis of Jean Mawhin}, eds. M.~Delgado,
J.~Lopez-Gomez, R.~Ortega and A.~Suarez, World Scientific Publishing, 2004.
%pp. 103--116.

\bibitem{Benoit}
E.~Beno\^it, J.~L.~Callot, F.~Diener and M.~Diener,
``Chasse au canard'',
\emph{Collectanea Mathematica}, vol. \textbf{31--32}, no. 1--3,
pp. 37--119, 1981--1982.

\bibitem{Kennedy}
M.~P.~Kennedy and L.~O.~Chua,
``Hysteresis in electronic circuits: A circuit theorist's perspective'',
\emph{International J. of Circuit Theory and Applications}, vol. \textbf{19},
pp.~471--515, 1991.

\bibitem{Bokhoven}
W.~M.~G.~van~Bokhoven,
``Piecewise linear analysis and simulation'',
in \emph{Circuit Analysis, Simulation and Design}, ed. A.~E.~Ruehli, 2, chapter 10,
Elsevier Science Publishers B. V. (North-Holland), 1987.

\bibitem{Fujisawa}
T.~Fujisawa and E.~S.~Kuh,
``Piecewise linear theory of nonlinear networks'',
\emph{SIAM J. Appl. Math.}, vol. \textbf{22}, pp. 307--328, 1972.

\bibitem{Porter}
B.~Porter, D.~L.~Hicks,
``Genetic robustification of digital model-following flight-controlsystems'',
\emph{IEEE Proceedings of the National Aerospace and Electronics Conference},
vol. \textbf{1}, pp. 556--563, 1994.

\bibitem{Ozkan}
L.~\"Ozkan, M.~V.~Kothare, C.~Georgakis,
``Control of a solution copolymerization reactor using piecewise linear models'',
\emph{IEEE Proceedings of the American Control Conference},
vol. \textbf{5}, pp. 3864--3869, 2002.

\bibitem{Mayeri}
E.~Mayeri,
``A relaxation oscillator description of the burst-generating mechanism
in the cardiac ganglion of the lobster, homarus americanus'',
\emph{J. Gen. Physiol.}, vol. \textbf{62}, pp. 473--488, 1973.

\bibitem{McKean}
H.~P.~McKean, ``Nagumo's equation'', \emph{Advances in Mathematics},
vol. \textbf{4}, pp. 209--223, 1970.

\bibitem{Sekikawa}
M.~Sekikawa, N.~Inaba, T.~Tsubouchi,
``Chaos via duck solution breakdown in a piecewise linear van der Pol
oscillator driven by an extremely small periodic perturbation'',
\emph{Physica D}, vol. \textbf{194}, pp. 227--249, 2004.

\bibitem{SIAM}
\emph{Singular Perturbations and Hysteresis},
eds. M.~P.~Mortell, R.~E.~O'Malley, A.~V.~Pokrovskii and V.~A.~Sobolev,
SIAM, 2005.

\bibitem{Arnold}
V.~I.~Arnold,  V.~S.~Afraimovich, Yu.~S.~Il'yashenko and L.~P.~Shil'nikov,
\emph{Theory of Bifurcations} (Dynamical Systems, vol. 5 of
Encyclopedia of Mathematical Sciences), ed. V.~Arnold, Springer, 1994.

\bibitem{Mishchenko}
E.~F.~Mishchenko, Yu.~S.~Kolesov, A.~Yu.~Kolesov and N.~Kh.~Rozov,
\emph{Asymptotic Methods in Singularly Perturbed Systems},
Plenum Press, 1995.

\bibitem{Deng}
B.~Deng, ``Food chain chaos with canard explosion,'' Chaos,
\textbf{14}, 1083--1092 (2004).

\bibitem{Zgli}
P.~Zgliczy\'nski, \emph{Fixed point index for iterations of maps,
topological horseshoe and chaos}, Topol. Methods Nonlinear Anal,
vol. \textbf{8} (1996), pp. 169--177.

\bibitem{Shadowing}
A.V.~Pokrovskii, \emph{Topological shadowing and
split-hyperbolicity}, Functional Differential Equations, special
issue dedicated to M.A.~Krasnosel'skii, vol. \textbf{4} (1997), pp.
335--360.

\bibitem{Homoclinic}
A.~Pokrovskii, S.~J.~Szybka and J.~G. McInerney, \emph{Topological
degree in locating homoclinic structures for discrete dynamical
systems}, Preprints of INS, UCC, Ireland, 01-001, 2001.

\bibitem{Cox}
E.~A.~Cox, M.~P.~Mortell, A.~Pokrovskii and O.~Rasskazov,
\emph{On Chaotic Wave Patterns in Periodically Forced Steady-State KdVB
and Extended KdVB Equations},
Proc. R. Soc. A., vol. \textbf{461} (2005), pp. 2857--2885.

\bibitem{Zgli2}
P.~Zgliczy\'nski, \emph{Computer assisted proof of the horseshoe
dynamics in the Henon map}, Random Comput. Dynam, vol. \textbf{5}
(1997), pp. 1--17.

\bibitem{Katok}
A.~E.~Katok and B.~Hasselblatt,
Introduction to the Modern Theory of Dynamical Systems,
Cambridge University Press, 1995.

\end{thebibliography}
\end{document}